\documentclass[12pt, reqno]{amsart}
\usepackage{amsmath}
\usepackage{amsfonts}
\usepackage{amssymb}
\usepackage[all,cmtip]{xy}           
\usepackage{bm}
\usepackage{bbm}
\usepackage{bbding}
\usepackage{txfonts}
\usepackage{amscd}
\usepackage{xspace}
\usepackage[shortlabels]{enumitem}
\usepackage{ifpdf}

\ifpdf
  \usepackage[colorlinks,final,backref=page,hyperindex]{hyperref}
\else
  \usepackage[colorlinks,final,backref=page,hyperindex,hypertex]{hyperref}
\fi
\usepackage{tikz}
\usepackage[active]{srcltx}

\usepackage{tikz-cd}
\usepackage{tikz}

\usepackage{pdfsync}    

\makeatletter

\newtheorem{thm}{Theorem}[section]
\newtheorem{lem}[thm]{Lemma}

\newtheorem{pro}[thm]{Proposition}

\newtheorem{question}{Question}
\theoremstyle{definition}
\newtheorem{ex}[thm]{Example}
\newtheorem{rmk}[thm]{Remark}
\newtheorem{defi}[thm]{Definition}

\newcommand{\nc}{\newcommand}
\newcommand{\delete}[1]{}

\nc{\mlabel}[1]{\label{#1}}  
\nc{\mcite}[1]{\cite{#1}}  
\nc{\mref}[1]{\ref{#1}}  
\nc{\meqref}[1]{\eqref{#1}}  
\nc{\mbibitem}[1]{\bibitem{#1}} 

\delete{
\nc{\mlabel}[1]{\label{#1}{\hfill \hspace{1cm}{\bf{{\ }\hfill(#1)}}}}
\nc{\mcite}[1]{\cite{#1}{{\bf{{\ }(#1)}}}}  
\nc{\mref}[1]{\ref{#1}{{\bf{{\ }(#1)}}}}  
\nc{\meqref}[1]{\eqref{#1}{{\bf{{\ }(#1)}}}}  
\nc{\mbibitem}[1]{\bibitem[\bf #1]{#1}} 
}

\nc{\tblue}[1]{\textcolor{blue}{#1}}
\nc{\tred}[1]{\textcolor{red}{#1}}
\nc{\tnb}[1]{\textcolor[rgb]{0.00,0.00,0.50}{#1}}
\nc{\tpurple}[1]{\textcolor{purple}{#1}}

\nc{\name}[1]{{\bf #1}}

\nc{\lc}{\lfloor}
\nc{\rc}{\rfloor}
\nc{\dirlim}{\varinjlim}

\nc{\bfk}{\mathbf{k}}
\nc{\invlim}{\displaystyle{\lim_{\longleftarrow}}\,}
\nc{\ot}{\otimes}
\nc{\im}{\mathrm{im}}
\nc{\free}[1]{\overline{#1}}

\nc{\tforall}{,\quad \forall}

\nc{\Ad}{\mathrm{Ad}}
\nc{\AD}{\overline{\mathrm{Ad}}}

\nc{\Aut}{\mathrm{Aut}}

\nc{\desc}{descendent\xspace}

\nc{\B}{\mathfrak{B}}\nc\A{\mathfrak{A}}

\nc{\BA}{{\mathbb A}} \nc{\CC}{{\mathbb C}} \nc{\DD}{{\mathbb D}} \nc{\EE}{{\mathbb E}} \nc{\FF}{{\mathbb F}} \nc{\GG}{{\mathbb G}} \nc{\HH}{{\mathbb H}} \nc{\LL}{{\mathbb L}} \nc{\NN}{{\mathbb N}} \nc{\KK}{{\mathbb K}} \nc{\PP}{{\mathbb P}} \nc{\QQ}{{\mathbb Q}} \nc{\RR}{{\mathbb R}} \nc{\TT}{{\mathbb T}} \nc{\VV}{{\mathbb V}} \nc{\ZZ}{{\mathbb Z}}

\nc{\cala}{{\mathcal A}} \nc{\calc}{{\mathcal C}} \nc{\cald}{{\mathcal D}} \nc{\cale}{{\mathcal E}} \nc{\calf}{{\mathcal F}} \nc{\calg}{{\mathcal G}} \nc{\calh}{{\mathcal H}} \nc{\cali}{{\mathcal I}} \nc{\call}{{\mathcal L}} \nc{\calm}{{\mathcal M}} \nc{\caln}{{\mathcal N}} \nc{\calo}{{\mathcal O}} \nc{\calp}{{\mathcal P}} \nc{\calr}{{\mathcal R}} \nc{\cals}{{\mathcal S}} \nc{\calt}{{\mathcal T}} \nc{\calw}{{\mathcal W}} \nc{\calk}{{\mathcal K}} \nc{\calx}{{\mathcal X}}
\nc{\calz}{{\mathcal Z}}
\nc{\WC}{\mathcal{WC}}

\nc{\bvp}[2]{\boxed{\begin{array}{l}#1\\#2\end{array}}}
\nc{\evl}{E} \nc{\cum}{{\textstyle \varint}} \nc{\mapmonoid}{\frakM}
\nc{\X}{X^{\pm 1}} \nc{\redx}{\text{Red}(X)} \nc{\bre}{{\rm bre}}\nc{\dep}{\mathrm{dep}}
\nc{\rbgw}{\frakA(X)} \nc{\rbgo}{\lc \, \rc} \nc{\lbar}{\overline}

\newcommand{\frakA}{\mathfrak A}

\newcommand{\frakM}{\mathfrak M}

\newcommand{\frakR}{\mathfrak R}

\nc{\aveo}{averaging operator\xspace}
\nc{\aveos}{averaging operators\xspace}
\nc{\aveg}{averaging group\xspace}
\nc{\avelg}{averaging Lie group\xspace}
\nc{\avegs}{averaging groups\xspace}
\nc{\paveo}{pointed averaging operator\xspace}
\nc{\paveos}{pointed averaging operators\xspace}
\nc{\paveg}{pointed averaging group\xspace}
\nc{\pavegs}{pointed averaging groups\xspace}
\nc\Aa{A} \nc\id{{\rm id}} \nc\tr{{\rm tr}}
\nc{\degp}{{\rm deg_{P}}}
\nc\rb{Rota-Baxter \xspace}
\nc\ave{averaging\xspace}
\nc\al{Lie algebras\xspace}
\nc\gr{Lie groups\xspace}
\nc{\rack}{\vartriangleright}

\begin{document}

\title[Averaging operators on groups and Hopf algebras]{Averaging operators on groups and Hopf algebras}

\author{Huhu Zhang}
\address{School of Mathematics and Statistics
	Yulin University,
	Yulin, Shaanxi 719000, China}
\email{huhuzhang@yulinu.edu.cn}

\author{Xing Gao$^{*}$}\thanks{*~Corresponding author}
\address{School of Mathematics and Statistics, Lanzhou University
Lanzhou, 730000, China;
Gansu Provincial Research Center for Basic Disciplines of Mathematics
and Statistics, Lanzhou, 730070, China
}
\email{gaoxing@lzu.edu.cn}
\date{\today}

\begin{abstract}
Rota-Baxter operators on groups were studied quite recently.
Motivated mainly by the fact that weight zero Rota-Baxter operators and averaging operators are Koszul dual to each other, we propose the concepts of averaging group and averaging Hopf algebra, and study relationships among them and the existing averaging Lie algebras.
We also show that an averaging group induces a disemigroup and a rack, respectively.
As the free object is one of the most significant objects in a category, we also construct explicitly the free averaging group on a set.
\end{abstract}

\makeatletter
\@namedef{subjclassname@2020}{\textup{2020} Mathematics Subject Classification}
\makeatother
\subjclass[2020]{
22E60, 
08B20, 
17B40 
16W99 %
}

\keywords{Operated groups, averaging groups, free object}

\maketitle

\vspace{-.9cm}

\tableofcontents

\vspace{-.9cm}

\allowdisplaybreaks

\section{Introduction}

This paper studies averaging operators on groups that are compatible with the Lie algebra
and Hopf algebra structures.

\subsection{Linear operators on algebras and operators on groups}
Linear operators, defined by the specific algebraic identities they satisfy, have been fundamental to the development of mathematical research across various fields. These operators, often characterized by their structural properties, continue to attract considerable attention due to their diverse applications and theoretical significance. Prominent examples include the derivative and integral operators in analysis, as well as the endomorphisms and automorphisms in the study of Galois theory. In addition to these well-known cases, there exists a rich spectrum of other operator types, such as difference operators, weighted differential operators, Rota-Baxter operators~\cite{Da,Gu,Gub,GK,Ro1}, Reynolds operators, averaging operators, and Nijenhuis operators. These, along with their multi-operator analogs, have emerged from a variety of mathematical domains, including geometry, probability theory, fluid mechanics, analysis~\cite{Bax}, combinatorics, differential equations~\cite{CGM,Ko,PS}, and mathematical physics~\cite{BGN,Ku,STS}. Their applications range widely, from the renormalization in quantum field theory to providing mechanical proofs of fundamental geometric theorems~\cite{CK,GGL,Kf,N,ZtGG,ZGG,ZGG2,ZGM}, showcasing their profound impact across both theoretical and applied mathematics.

In an effort to unify these various operator identities into a coherent framework, Rota~\cite{Ro2} proposed the {\bf Rota's Classification Problem}, which can be phrased as:
\begin{quote}
finding all possible \underline{algebraic identities} that can be satisfied by a linear operator on an \underline{algebra}.
\end{quote}
This problem, which has become a central theme in the study of linear operators, seeks to understand the full extent of algebraic structures that such operators can generate. These have facilitated the discovery of previously unknown operator identities, particularly on associative algebras~\mcite{GG,GGZh,Guop,GSZ} and Lie algebras~\mcite{ZhGG}, revealing new layers of structure in both classical and modern algebra.

While the investigation of operator identities on algebras remains a central focus, recent years have seen a notable resurgence of interest in the study of operators on groups. Despite the fact that groups lack the linear structure of vector spaces or algebras, the study of operators on groups has a long and storied history. In the early 20th century, Emmy Noether and her school made significant contributions to the theory of groups with automorphisms, incorporating this concept into her original formulation of the three Noetherian isomorphism theorems, which are foundational to modern algebra.
In the Bourbaki treatment of group theory~\cite{Bou}, the idea of groups equipped with endomorphisms was introduced at the very outset, a concept that would later be expanded to define modules, in much the same way that rings with operators are used to define algebras. This analogy has proven fruitful in bridging ideas between group theory and other branches of algebra.

The continuing exploration of operators, whether on algebras or groups, is a testament to the central role they play in the evolution of mathematical thought. Their study has led to significant breakthroughs in both abstract theory and practical applications, and the ongoing research in this area promises to reveal even deeper connections between algebra, geometry, and physics.

\subsection{Averaging operators}
Averaging operators first introduced in the work of Reynolds in 1895 in connection with the theory of turbulence~\mcite{Rey},
and have been studied in various areas.
The averaging operator provides a useful tool on the construction of supergravity theories~\mcite{NS} and higher gauge theories~\mcite{BDH}, which is called the embedding tensor in  physics. Kotov
and Strobl used embedding tensors to construct tensor hierarchies, which show us the possible
mathematical properties of embedding tensors from the physics point of view~\mcite{KS}.

Brainerd~\mcite{Br} considered the conditions under which an averaging operator can be represented as an integration on the abstract analogue of the ring of real valued measurable function.
Aguiar~\mcite{Ag} showed that a diassociative algebra can be derived from an averaging associative algebra by defining two new operations $x \dashv y := xP(y)$ and $x \vdash y := P(x)y$. An analogue process gives a Leibniz algebra from an averaging Lie algebra by defining a new operation ${x, y} := [P(x), y]$.
In general, Kolesnikov and his coauthors introduced  notions di-Var-algebra and tri-Var-algebra~\cite{GKo,KV}.
Their notions also apply to not necessarily binary operads.
Alternatively, an averaging operator was defined on any binary operad and this kind of process was systematically studied in~\mcite{PBGN,PBGN1} by relating the averaging actions to a special construction of binary operads called duplicators.
In~\mcite{GP}, the authors studied averaging operators from an algebraic and combinatorial point of view.
Das~\mcite{Das} introduced the notion of a pointed averaging operator on a group, which can be induced from our presently studied averaging operators on groups.

\subsection{Rota-Baxter operators on groups}
Recently,  in the seminal work~\mcite{GLS}, Guo, Lang and Sheng introduced and studied Rota-Baxter operators of weight $\pm1$ on groups, motivated by the classical Yang-Baxter equation and Poisson geometry~\mcite{Mk,STS}. The differentiation of smooth Rota-Baxter operators on Lie groups gives Rota-Baxter operators on the corresponding Lie algebras. As an application, the fundamental factorization of a Lie group originally obtained indirectly from locally integrating a factorization of a Lie algebra now comes directly from a global factorization of the Lie group~\mcite{RS1,RS2,STS}, equipped with the Rota-Baxter operator. Incidentally, the formal inverse of a Rota-Baxter operator on a Lie group is none other than the crossed homomorphism on the Lie group, and differentiates to a differential operator on the Lie algebra of left-translation-invariant vector fields.

After~\cite{GLS}, there has been a boom of studies on the theory and applications of Rota-Baxter and differential operators on groups. The free Rota-Baxter group and free differential group were explicitly constructed in~\mcite{GGLZ}.
General properties of Rota-Baxter groups, especially extensions of Rota-Baxter groups and Rota-Baxter operators on sporadic simple groups, were studied in~\mcite{BG}. It was shown that Rota-Baxter groups give rise to braces and the more general skew left braces in quantum Yang-Baxter equation~\mcite{BG2}.
In~\mcite{CS}, a characterization of the gamma functions on a group that come from Rota-Baxter operators as in~\mcite{BG2} was given in terms of the vanishing of a certain element in a suitable second cohomology group.
Rota-Baxter operators on Clifford semigroups were introduced as a useful tool for obtaining dual weak braces~\mcite{CMS}.
In~\mcite{BNY}, Rota-Baxter groups and skew braces were further studied.
Quite recently, two kinds of weight zero Rota-Baxter and differential operators on groups were introduced and studied in~\mcite{GGH, GGHZ}, respectively.

A cohomological theory of Rota-Baxter Lie groups was developed which differentiates to the existing cohomological theory of Rota-Baxter Lie algebras~\mcite{JSZ}.
A one-to-one correspondence was established between factorizable Poisson Lie groups and quadratic Rota-Baxter Lie groups~\mcite{LS}.
A Rota-Baxter operator on a cocommutative Hopf algebra was introduced and studied~\mcite{Gon}, generalizing the notions of a Rota-Baxter operator on a group and a Rota-Baxter operator of weight 1 on a Lie algebra.
Post-groups and pre-groups were introduced~\mcite{BGST} that can be derived from Rota-Baxter groups and capture the extra structures on (Lie-)Butcher groups in numerical integration, braces, Yang-Baxter equation and post-Hopf algebras~\mcite{BG2,CJO,ESS,LYZ,ML,MW}.

\subsection{Motivation and outline of the present paper}

Our motivation to propose the concept of averaging operator on a (Lie) group is twofold.
One is from the Koszul duality and the other is from taking tangent space.

The action of the Rota-Baxter operator on a binary quadratic operad splits the operad~\cite{BBGN},
and the action of the averaging operator on a binary quadratic operad duplicates the operad~\cite{PBGN}.
Given that the splitting and duplication processes are in Koszul duality~\cite{PBGN1}, the Rota-Baxter operator and the averaging operator can be considered as
Koszul duals of each other in this context.
For example, the Koszul dual of dendriform algebras (induced by weight zero Rota-Baxter associative algebras) is diassociative algebras (induced by averaging associative algebras);
the Koszul dual of preLie algebras (induced by weight zero Rota-Baxter Lie algebras)  is Leibniz algebras (induced by averaging commutative algebras).
Recently, Guo et al.~\mcite{GGH,GGHZ} first introduced  two kinds of weight zero Rota-Baxter operators on groups.
One is the limit weight zero Rota-Baxter operators on groups. The other is
the limit-weighted Rota-Baxter operators on limit-abelian groups.
Under the framework of Lie groups, taking tangent space of both of them and then restricting to some special cases give the weight zero Rota-Baxter Lie algebras.
Meanwhile, the new notation of pre-Lie group was also initiated
such that the following relations hold.
$$
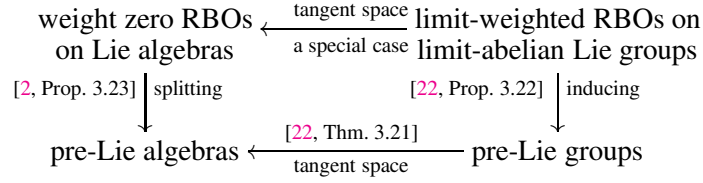
\small{
\xymatrix{
    \txt{weight zero RBOs \\ on Lie algebras }&  & \txt{limit-weighted RBOs on\\ limit-abelian Lie groups}\\
    \text{pre-\al }& &\text{pre-\gr }\\
    \ar@{->}"1,3";"1,1"_{\text{tangent space}}^{\text{a special case}}
     \ar@{->}"2,3";"2,1"^{\text{tangent space}}_{\text{\cite[Thm.~3.21]{GGHZ}}}
      \ar@{->}"1,1";"2,1"^{\text{splitting}}_{\text{\cite[Prop.~3.23]{BBGN}}}
       \ar@{->}"1,3";"2,3"^{\text{inducing }}_{\text{ \cite[Prop.~3.22]{GGHZ} }}
}
}
$$\vskip-0.8cm
\noindent Here and in Figures below, we abbreviate Rota-Baxter operator and Rota-Baxter as RBO and RB, respectively.
Also, Kinyon showed that every Leibniz algebra is the tangent Leibniz algebra of a Lie rack~\cite[Lemma 5.4]{Ki}.
Therefore, it is natural to consider the following question.

\begin{question}
Propose and study averaging operators on $($Lie$)$ groups such that
\begin{enumerate}
\item from the viewpoint of Koszul duality, the expected relations in Figure~$\mref{fig:kdual1}$ hold.

\item from the viewpoint of taking tangent space, the expected relations in Figure~$\mref{fig:tspace}$ are valid.
\end{enumerate}
\end{question}

{\bf The outline.} Our main goal of the present paper is to address the above question.
The outline is as follows.
Section~\mref{sec:opgp} is devoted to recalling some concepts and results on operated groups.
In particular, the explicit construction of the free operated group is recalled, which will be employed to construct free averaging groups.
In Section~\mref{sec:avegp},  we first propose the concept of averaging group. Then we show that an averaging group induces a disemigroup (Proposition~\ref{pro:digr}) and a rack (Proposition~\ref{pro:rac}), respectively.
The differentiation of a smooth averaging operator on a Lie group gives an
averaging operator on the corresponding Lie algebra (Theorem~\mref{thm:avega}).
Based on proposing the concept of averaging Hopf algebra, we prove that
an operator on a group is an
averaging operator if and only if its linear extension on the corresponding group Hopf algebra is an averaging operator
(Theorem~\mref{thm:hopf}).
Finally in Section~\mref{sec:rbgp}, motivated by the idea of rewriting systems and Gr\"{o}bner-Shirshov bases, we collect a special subset of elements in the free operated
group and define a multiplication and an operator on it to form an averaging group (Theorem~\mref{thm:avegp}). Based on this, we further prove that this averaging group is the required free averaging group on a set (Theorem~\mref{thm:freea}).
\begin{figure}[htbp]\vskip-0.6cm
\caption{Relations from the viewpoint of Koszul duality}
\mlabel{fig:kdual1}
\begin{displaymath}
{\small\xymatrix{
\txt{weight zero \\ RB  groups}& &\text{groups} && \text{\ave groups}\\
\txt{weight zero RB \\ associative algebras}&& \txt{associative \\ algebras} & &\txt{\ave \\ associative algebras}\\
\txt{dendriform algebras}&&  && \text{diassociative algebras}\\
\ar@{->}"1,3";"1,1"_{\quad\quad\text{weight zero RBO}}
\ar@{->}"1,3";"1,5"^{\text{averaging operator }}
\ar@{->}"2,3";"2,1"_{\quad\quad\text{weight zero RBO }}
\ar@{->}"2,3";"2,5"^{\text{averaging operator }}
\ar@{->}"1,1";"2,1"|-{\text{linear extension}}
\ar@{->}"1,3";"2,3"|-{\text{linear extension}}
\ar@{->}"1,5";"2,5"|-{\text{linear extension}}
\ar@{->}"2,1";"3,1"|-{\text{splitting}~~\text{\cite[Coro.~5.2]{BBGN}}}
\ar@{->}"2,5";"3,5"|-{\text{replicating}~~\text{\cite[Thm.~4.3]{PBGN}}}
\ar@{<->}"3,1";"3,5"^-{\text{Koszul dual}~~\text{\cite{Lod}}}
}}
\end{displaymath}
\vskip-2cm
\end{figure}
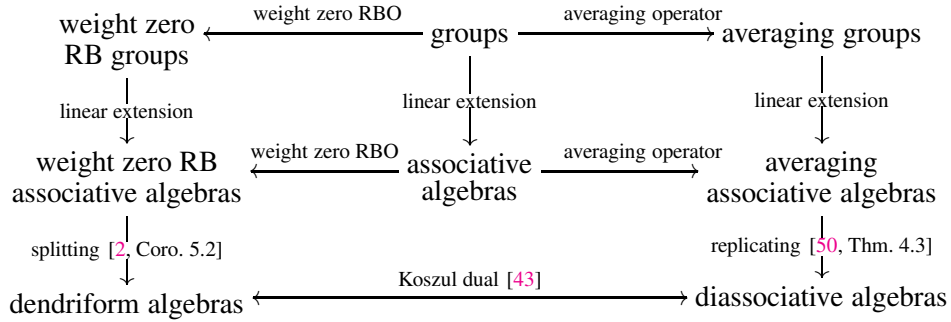
\begin{figure}[htbp]
\caption{Relations from the viewpoint of taking tangent space}
\mlabel{fig:tspace}
\begin{displaymath}
\small{
\xymatrix{
    \text{\ave \al}&  & \text{\ave \gr}\\
    \text{Leibniz algebras}& &\text{Lie racks }\\
    \ar@{->}"1,3";"1,1"_{\text{tangent space}}^{\text{Thm.~\ref{thm:avega}}}
     \ar@{->}"2,3";"2,1"^{\text{tangent space}}_{\text{\cite[Lemma 5.4]{Ki}}}
      \ar@{->}"1,1";"2,1"^{\text{\cite[Coro.~4.4]{PBGN}}}_{{\rm inducing}}
       \ar@{->}"1,3";"2,3"^{{\rm Prop.~\ref{pro:rac}}}_{{\rm inducing}}
}
}
\end{displaymath}
\vskip-1.5cm
\end{figure}
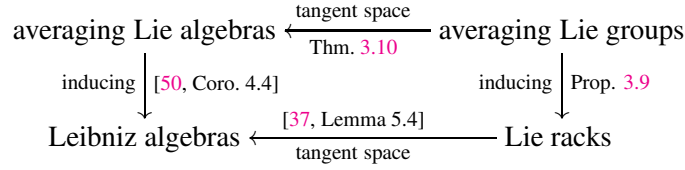

\section{Operated groups}
\mlabel{sec:opgp}
In this section, we recall some notations and results of operated groups, which will be used later.
Let us begin with the following concept.

\begin{defi}\mcite{GGLZ}
An {\bf operated group} is a group $G$ together with a map $P:G\to G$.
A {\bf morphism} from an operated group $(G,P)$ to another one $(H,Q)$ is a group homomorphism $\varphi:G\to H$ such that $Q\circ \varphi=\varphi \circ P$.
\end{defi}

We recall some examples of operated groups.

\begin{ex} Let $G$ be a group.
\begin{enumerate}

\item For each $g\in G$, the pair $(G, \Ad_g)$ is an operated group, where $\Ad_g$ is
the adjoint action
$$\Ad_g:G\longrightarrow G, \quad h \mapsto \Ad_g h:= ghg^{-1}.$$

\item The pair $(G, \varphi)$ is an operated group, where $\varphi:G\rightarrow G$ is a group endomorphism.

 \item The differential group $(G, \mathfrak{D})$ of weight 1 (resp. weight -1)~\cite{GLS} is an operated group, which arises from crossed homomorphisms~\mcite{Se} and satisfies
 $$\mathfrak{D}(gh)=\mathfrak{D}(g)\Ad_g \mathfrak{D}(h) \quad  \Big(\text{resp. } \mathfrak{D}(gh)=(\Ad_g \mathfrak{D}(h))\mathfrak{D}(g)\Big) \tforall g, h\in G. $$

 \item  The Rota-Baxter group $(G, \B)$ of weight 1 (resp. weight -1)~\mcite{GLS} is an operated group with the Rota-Baxter relation given by
 $$\B(g)\B(h)=\B(g\Ad_{\B(g)} h) \quad \Big(\text{resp. } \B(g)\B(h)=\B( (\Ad_{\B(g)} h) g) \Big)\tforall~ g,h\in G,$$
 coming from studies of integrable systems and Hopf algebras.
\end{enumerate}
\end{ex}

The notion of free operated groups can be defined by a universal property.
We recall an explicit construction of the free operated group on a set by bracketed words~\mcite{GGLZ}.

\begin{defi}
Let $X$ be a set. The \name{free operated group} on $X$ is an operated group $(F(X),P_X)$ together with a map $j_X:X\to F(X)$ with the property that, for each operated group $(G,P)$ and each set map $\varphi:X\to G$, there is a unique homomorphism $\free{\varphi}:(F(X),P_X)\to (G,P)$ of operated groups such that $\free{\varphi}\circ j_X= \varphi$. 	
\end{defi}

We first give some general notations. For any set $Y$, let
$$Y^{-1} :=\{y^{-1}\mid y\in Y\},\quad  \lc Y \rc :=\{\lc y\rc\mid y\in Y\} $$
denote two distinct copies of $Y$ that are disjoint from $Y$.
Let $G(Y)$ denote the free group generated by $Y$. The $G(Y)$ consists of the identity $1$ and reduced words in the letter set $Y\sqcup Y^{-1}$, where a reduced word is an element of the form
$ w=w_1\cdots w_k,$
with $w_i\in Y\sqcup Y^{-1}$ so that no adjacent letters $w_i$ and $w_{i+1}$ are of the form $y y^{-1}$ for a $y\in Y\sqcup Y^{-1}$.

Now fix a set $X.$ Define a directed system of free groups $\calg_n:=\calg_n(X), n\geq 0,$ by a recursion. First denote
$\calg_0: = G(X).$
Next for given $k\geq0$, assume that $ \calg_n$  has been defined for $n < k$, such that
\begin{enumerate}
  \item $\calg_{n+1}:=G(X\sqcup \lc \calg_{n}\rc)$.
  \item there are natural injections $\calg_{n} \hookrightarrow \calg_{n+1}$ of groups.
\end{enumerate}
We then recursively define
$$\calg_{k+1}:=G(X\sqcup \lc \calg_{k}\rc).$$
Further the injection $\calg_{k-1} \hookrightarrow \calg_{k}$  induces the set injection
$$X\sqcup \lc \calg_{k-1}\rc \hookrightarrow X\sqcup \lc \calg_{k}\rc.$$
Then the functoriality of taking the free group extends to the injection
 $\calg_{k} \hookrightarrow \calg_{k+1}$
of free groups.

Finally, we define the direct limit of the directed system
$$ \calg:=\calg(X):=\dirlim \calg_n = \bigcup_{n\geq 0} \calg_n.$$
It is still a group which naturally contains $\calg_n, n\geq 0,$ as subgroups via the structural injective group homomorphisms
\begin{equation*}
	i_n:\calg_n\to \calg.
	\mlabel{eq:nind}
\end{equation*}

We also recall some basic properties and notations of elements of $\calg$.
\begin{enumerate}
\item For $w\in \calg$, we have $w\in \calg_n$ for some $n\geq 0$. Thus $\lc w\rc$ is in $\calg_{n+1}$ and so in $\calg$. Notice that $\lc w\rc$ is independent of the choice of $n$ such that $w$ is in $\calg_n$.
Hence $$P_X:=\lc ~\rc: \calg(X)\to \calg(X), \quad w\mapsto \lc w\rc$$ is an operator on the group $\calg(X)$.

\item For $n\geq1$, denote by $\lc w \rc^{(n)}$ the $n$-th iteration on  $w\in \calg$ of the operator $\lc~\rc$.

\item Denote
$$X^{\pm 1}:= X \sqcup X^{-1},\,\lc \calg\rc^{-1}:= \{w^{-1} \mid w\in \lc \calg\rc\}\,\text{ and }\,\lc \calg\rc^{\pm 1}:= \lc \calg\rc \sqcup \lc \calg\rc^{-1}. $$

\item Every element $w\in \calg$ can be uniquely written in the form
\begin{equation}
\mlabel{eq:frbg}
w = w_1\cdots w_k\, \text{ with }\, w_i\in X^{\pm 1} \sqcup \lc \calg\rc^{\pm 1},
\end{equation}
with no adjacent factors being the inverse of each other. The factorization is called the \name{standard factorization} of $w$.
Elements of $X^{\pm 1} \sqcup \lc \calg\rc^{\pm 1}$ are called {\bf indecomposable.} Define the \name{breadth} of $w$ to be $\bre(w):=k$.

\item Define the  \name{(operator) degree} $\degp(w)$ of $w$ to be  the total number of occurrences of the operator $\lc ~\rc$ in $w\in \calg$.

\item For $w$ in~(\mref{eq:frbg}), define the  \name{depth} of $w$ to be
$$\dep(w):= \max \{ \dep(w_i) \mid 1\leq i\leq k\}, \,\text{ where }\, \dep(w_i) := \min\{n \mid w_i\in \calg_n\setminus \calg_{n-1} \}.$$

\end{enumerate}

\begin{lem}\cite{GGLZ} \mlabel{thm:freeopgp}
	The operated group $(\calg(X),P_X)$, together with the the natural inclusion $j_X:X\to \calg(X)$, is the free operated group on $X$.
\end{lem}

\section{Averaging operators on groups, Lie algebras and Hopf algebras}
\mlabel{sec:avegp}
In this section, we first propose the concept of averaging group and give some examples and basic properties of averaging groups. Then we prove that an averaging group induces a rack and that the differentiation of a smooth averaging operator on a Lie group gives an averaging operator on the corresponding Lie algebra.

\subsection{Averaging groups and averaging Lie algebras}
The averaging operators on associative algebras were already implicitly studied by O. Reynolds~\mcite{Rey} in turbulence theory under the disguise of a Reynolds operator.
Similar to the case of associative algebras, averaging operators on Lie algebras were studied by Aguiar~\mcite{Ag}.

\begin{defi}
\begin{enumerate}
\item An {\bf averaging associative algebra} is an associative algebra $R$ with a linear operator $P: R\to R$ such that
$$P(a)P(b)=P(P(a)b)=P(aP(b)) \tforall a,b\in R.$$

\item  An {\bf averaging Lie algebra} is a Lie algebra $(L, [\,,\,])$ with a linear operator $A: L\to L$ such that
$$[A(a),A(b)]=A([A(a),b]) \tforall a,b\in L.$$
\end{enumerate}
\end{defi}

By the antisymmetric  property of the Lie bracket, we have
$$[A(a),A(b)]=-[A(b),A(a)]=-A([A(b),a])=A([a, A(b)]).$$
Thus an averaging Lie algebra $(L,[\,,\,], A)$ satisfies
\begin{equation}\mlabel{eq:alie}
[A(a),A(b)]=A([A(a),b])=A([A(a),b]) \tforall a,b\in L.
\end{equation}

Motivated by  ~(\mref{eq:alie}), we propose the following main concept studied in the present paper.

\begin{defi}
\mlabel{de:savegp}
A map $\A$ on a group $G$ is called an \name{ \aveo } if
	\begin{equation}
		\mlabel{eq:avegp}
\A(g)\A(h)=\A(\A(g)h)=\A(g\A(h)) \tforall g,h \in G.
	\end{equation}
Then the pair $(G,\A)$ is called an \name{\aveg }.
Further, if $G$ is a Lie group and $\A:G\to G$ is smooth, then
$(G,\A)$ is called an \name{\avelg}.
\end{defi}

\begin{rmk}\mlabel{re:paveo}
\begin{enumerate}
\item For an averaging group $(G, \A)$, the group algebra $\bfk [G]$ with the linear extension of $\A$ to $\A:\bfk [G]\rightarrow \bfk[G]$ is an averaging associative algebra.

\item Das~\mcite{Das} defined the concept of \name{\paveo} arising from studies of rack structures and Hopf algebras, which is a group $G$ together with a map $\A:G\rightarrow G$ such that
$\A(e) = e$ for the identity element $e\in G$ and
	\begin{equation*}
\Ad_{\A(g)}\A(h) = \A(\Ad_{\A(g)}h) \tforall g,h \in G.
	\end{equation*}
It should be noted that the defining equation~(\mref{eq:avegp}) is quite different from the above equation.
\end{enumerate}
\end{rmk}

The first example is that a group $G$  equipped with the identity map on $G$ is an averaging group.
We display the following other classes of examples for \avegs.

\begin{pro}
Let $G$ be a group.
\begin{enumerate}
\item  Let $z$ be a fixed element in the center of $G$. Define
 $$\A_z:G\longrightarrow G, \quad g\mapsto zg.$$
Then the pair $(G, \A_z)$ is an averaging group.\mlabel{item:cavo}

\item The group $G$ together with an idempotent group endomorphism $\A:G\to G$ is an averaging group.\mlabel{item:ieavo}
\end{enumerate}
\end{pro}

\begin{proof}
The results respectively follow from
$$\A_z(\A_z(g)h)=zzgh=zgzh=\A_z(g)\A_z(h)= z(gzh) =\A_z(g\A_z(h))$$
and
$$\A(\A(g)h) = \A^2(g)\A(h) = \A(g)\A(h) =\A(g)\A^2(h)=\A(g\A(h)),$$
where $g,h\in G$.
\end{proof}

\begin{pro}
 Let $G$ be a group with two commutative \aveos $\A_1$ and $\A_2$.
Then the composition operator $\A_1\circ \A_2$ is also an \aveo on $G$.
In particular, the operator $(\A_1)^n$ is also  an \aveo on $G$ for any $n\geq1$.
\end{pro}

\begin{proof}
Let $g,h\in G$. By~\meqref{eq:avegp} and $\A_1\circ \A_2 = \A_2 \circ \A_1$,
{\small\begin{eqnarray*}
\Big( \A_1\circ \A_2 (g)\Big) \Big(\A_1\circ \A_2(h)\Big)
&=&\A_1\Big(\A_1\big(\A_2 (g))\A_2(h)\Big)\\
&=&\A_1\Big(\A_2\big(\A_1 (g))\A_2(h)\Big)\\
&=&\A_1 \circ \A_2\Big(\A_2(\A_1 (g))h\Big)\\
&=&\A_1\circ \A_2\Big(\A_2\circ \A_1(g)h \Big).
\end{eqnarray*}}
With the same argument,
$$\Big(\A_1\circ \A_2(g)\Big) \Big(\A_1\circ \A_2(h)\Big) =\A_1\circ \A_2(g\A_1\circ \A_2(h)).$$
The second part follows from the repeated use of the first one.
\end{proof}

The following result gives some basic properties of averaging groups.

\begin{pro}
Let $(G,\A)$ be an \aveg such that $\A(e)=e$ for the identity element $e$.
Then
\begin{enumerate}
\item $\A$ is idempotent, that is, $\A^2=\A$. \mlabel{item:idet}

\item   \mlabel{item:invp} $\A$ preserves the inverses of elements in $\im \A$, that is,
  \begin{equation}\mlabel{eq:pinv}
   \A(g)^{-1}=\A\Big(\A(g)^{-1}\Big)\tforall g\in G.
  \end{equation}

  \item $\A$ is a \paveo defined in Remark~\mref{re:paveo}.\mlabel{item:paveo}
\end{enumerate}
\mlabel{pro:prei}
\end{pro}

\begin{proof}
\meqref{item:idet} Using  ~\meqref{eq:avegp} and $\A(e)=e$,
$$\A(g)=\A(g)\A(e)=\A(\A(g)e)=\A^2(g)\tforall g\in G.$$
\par
\noindent
\meqref{item:invp} Once again, by  ~\meqref{eq:avegp} and $\A(e)=e$,
$$\A(g)\A\Big(\A(g)^{-1}\Big)=\A\Big(\A(g)\A(g)^{-1}\Big)=\A(e)=e\tforall g\in G,$$
and so $\A(g)^{-1}=\A\Big(\A(g)^{-1}\Big)$.
\par
\noindent
\meqref{item:paveo} By the definition of \paveos, we only need to prove that
$$\Ad_{\A(g)}\A(h) = \A(\Ad_{\A(g)}h)\tforall g,h\in G.$$
By Item~\meqref{item:invp} and  ~\meqref{eq:avegp},
\begin{eqnarray*}
\Ad_{\A(g)}\A(h) &=&\A(g)\A(h)\A(g)^{-1}\\
&=&\A(g)\A(h)\A\Big(\A(g)^{-1}\Big)\\
&=&\A(\A(g)h)\A\Big(\A(g)^{-1}\Big)\\
&=&\A\Big(\A(g)h\A\big(\A(g)^{-1}\big)\Big)\\
&=&\A\Big(\A(g)h\A(g)^{-1}\Big)\\
&=&\A(\Ad_{\A(g)}h),
\end{eqnarray*}
as required.
\end{proof}

The concept of disemigroup was introduced in~\cite{Lod}, whose linear extension $\bfk[G]$ is a diassociative algebra.
We are going to show that an \aveg induces a disemigroup.

\begin{defi}
A set $G$, equipped with two multiplications $\dashv, \vdash:G\times G \to G,$
is called a {\bf disemigroup} if
  \begin{eqnarray*}
(f\dashv g)\dashv h&=&f\dashv (g\dashv h),\\
(f\dashv g)\dashv h&=&f\dashv (g\vdash h),\\
(f\vdash g)\dashv h&=&f\vdash (g\dashv h),\\
(f\dashv g)\vdash h&=&f\vdash(g\vdash h),\\
(f\vdash g)\vdash h&=&f\vdash(g\vdash h) \tforall f,g,h\in G.
  \end{eqnarray*}
Furthere if there is an element $e\in G$ such that
$$g\dashv e=g=e\vdash g\tforall{g\in G},$$
then it is called a {\bf dimonoid}.
\end{defi}

\begin{pro}\mlabel{pro:digr}
Let $(G,\A)$ be an \aveg. Define
$$\dashv:G\times G \to G,\quad (g,h)\mapsto  g\dashv h:=g\A(h),$$
$$\vdash:G\times G \to G, \quad (g,h)\mapsto  g\vdash h:=\A(g)h.$$
Then the triple $(G, \dashv, \vdash)$ is a disemigroup.
Moreover, if $\A(e)=e$, then
$(G, \dashv, \vdash)$ is a dimonoid.
\end{pro}

\begin{proof}
Let $f, g, h\in G$.
The first part follows from that
\begin{eqnarray*}
(f\dashv g)\dashv h&=&\big(f\A(g)\big)\A(h)=f\big(\A(g)\A(h)\big)=f\A\big(g\A(h)\big)=f\dashv (g\dashv h),\\
(f\dashv g)\dashv h&=&\big(f\A(g)\big)\A(h)=f\big(\A(g)\A(h)\big)=f\A\big(\A(g)h\big)=f\dashv (g\vdash h),\\
(f\vdash g)\dashv h&=&\big(\A(f)g\big)\A(h)=\A(f)\big(g\A(h)\big)=f\vdash (g\dashv h),\\
(f\dashv g)\vdash h&=&\A\big(f\A(g)\big)h=\big(\A(f)\A(g)\big)h=\A(f)\big(\A(g)h\big)=f\vdash(g\vdash h),\\
(f\vdash g)\vdash h&=&\A\big(\A(f)g\big)h=\big(\A(f)\A(g)\big)h=\A(f)\big(\A(g)h\big)=f\vdash(g\vdash h).
\end{eqnarray*}
The second one is from $$g\dashv e=g\A(e)=g=\A(e)g=e\vdash g.$$
Hence the result holds.
\end{proof}

An \aveg with a mild condition induces a rack.
Notice that the tangent space of a Lie rack~\mcite{Ki} is a Leibniz algebra, similar to the case of Lie groups and Lie algebras.

\begin{pro}\mlabel{pro:rac}
 Let $(G,\A)$ be an \aveg with $\A(e)=e$. Define a multiplication on $G$:
 $$\rack :G\times G \to G, \quad (g,h)\mapsto  g\rack h:=\A(g)h \A(g)^{-1}.$$
Then
  \begin{eqnarray*}
f\rack (g\rack h)=(f\rack g)\rack (f\rack h)\tforall f,g,h\in G,
  \end{eqnarray*}
and for any $g\in G$, the map
$$L_g:G\to G, \quad h\mapsto g\rack  h$$  is a bijection.
In  this case, the pair $(G, \rack)$ is called a {\bf  rack}.
\end{pro}

\begin{proof}
Let $f,g,h\in G$.
By Proposition~\mref{pro:prei}, $\A$ is a pointed averaging operator, that is
\begin{equation}\mlabel{eq:pavg}
\A(g)\A(h)\A(g)^{-1}=\A\Big(\A(g)h\A(g)^{-1}\Big).
\end{equation}
Further by the definition of $\rack$,
\begin{eqnarray*}
f\rack (g\rack h)&=&f\rack \Big(\A(g)h \A(g)^{-1}\Big)\\
&=& \A(f) \Big(\A(g)h \A(g)^{-1}\Big)\A(f)^{-1} \\
&=& \A(f) \A(g)h \big(\A(f)\A(g)\big)^{-1} \\
&=& \A(f) \A(g)\A(f)^{-1}\A(f) h \big(\A(f)\A(g)\A(f)^{-1}\A(f)\big)^{-1} \\
&\overset{(\ref{eq:pavg})}{=}&\A\Big(\A(f)g \A(f)^{-1}\Big)\A(f) h \Big(\A\Big(\A(f)g \A(f)^{-1}\Big)\A(f)\Big)^{-1}\\
&=&\A\Big(\A(f)g \A(f)^{-1}\Big)\Big(\A(f)h \A(f)^{-1}\Big)\A\Big(\A(f)g \A(f)^{-1}\Big)^{-1}\\
&=&\Big(\A(f)g \A(f)^{-1}\Big)\circ\Big(\A(f)h \A(f)^{-1}\Big)\\
&=&(f\rack g)\rack (f\rack h).
\end{eqnarray*}
Moreover, for any $g\in G$, define a map
$$L_g^{-1}:G\to G, \quad h\mapsto \A(g)^{-1}h \A(g).$$
Then
$$L_g\circ L_g^{-1}(h)=L_g\Big(\A(g)^{-1}h \A(g)\Big)=\A(g)\A(g)^{-1}h \A(g)\A(g)^{-1}=\id_G(h)$$
and
$$L_g^{-1}\circ L_g(h)=L_g\Big(\A(g)h \A(g)^{-1}\Big)=\A(g)^{-1}\A(g)h \A(g)^{-1}\A(g)=\id_G(h)\tforall h\in G.$$
Thus the map $L_g$ is a bijection.
\end{proof}

In analogy to the fact that Lie groups are integrations of Lie algebras, we now show that the differentiation of a smooth averaging operator on a Lie group gives an averaging operator on the corresponding Lie algebra.

\begin{thm}
Let $(G, \A)$ be an \avelg  such that $\A(e)=e$. Denote by $\mathfrak{g} = T_e G$  the Lie algebra of $G$. Define
$\Aa:=\A_{\ast e}: g\to g$
to be the tangent map of $\A$ at the identity element $e$. Then $(\mathfrak{g}, \Aa)$ is an averaging Lie algebra.
\mlabel{thm:avega}
\end{thm}

\begin{proof}
Denote by $\exp : \mathfrak{g} \to G$ the exponential map.
For sufficiently small $t$, we have
\begin{equation}\mlabel{eq:tang}
\frac{d}{dt}\Big|_{t=0}\A(\exp^{tu})=\frac{d}{dt}\Big|_{t=0}\exp^{t\Aa(u)}=\Aa(u)\tforall u\in \mathfrak{g}.
\end{equation}
Notice that the Lie bracket on $\mathfrak{g}$ is given by
\begin{equation}\mlabel{eq:lie}
[u,v]=\frac{d^2}{dtds}\Big|_{t,s=0}\exp^{tu}\exp^{sv}\exp^{-tu}\tforall u,v\in \mathfrak{g}.
\end{equation}
For any $u,v\in \mathfrak{g}$,
\begin{eqnarray*}
[\Aa(u),\Aa(v)]&=&\frac{d^2}{dt\,ds}\Big|_{t,s=0}\exp^{t\Aa(u)}\exp^{s\Aa(v)}\exp^{-t\Aa(u)} \hspace{2cm}(\text{by  ~\meqref{eq:lie}})\\
&=&\frac{d^2}{dt\,ds}\Big|_{t,s=0}\A\Big(\exp^{tu}\Big)
\A\Big(\exp^{sv}\Big)\A\Big(\exp^{tu}\Big)^{-1}\hspace{1cm}(\text{by  ~\meqref{eq:tang}})\\
&=&\frac{d^2}{dt\,ds}\Big|_{t,s=0}\A\Big(\exp^{tu}\Big)
\A\Big(\exp^{sv}\Big)\A\Big(\A\big(\exp^{tu}\big)^{-1}\Big)\hspace{1cm}(\text{by  ~\meqref{eq:pinv}})\\
&=&\frac{d^2}{dt\,ds}\Big|_{t,s=0}\A\Big(\A\big(\exp^{tu}\big)
\exp^{sv}\Big)\A\Big(\A\big(\exp^{tu}\big)^{-1}\Big)\hspace{1cm}(\text{by  ~\meqref{eq:avegp}})\\
&=&\frac{d^2}{dt\,ds}\Big|_{t,s=0}\A\bigg(\A\big(\exp^{tu}\big)\exp^{sv}
\A\Big(\A\big(\exp^{tu}\big)^{-1}\Big)\bigg)\hspace{1cm}(\text{by  ~\meqref{eq:avegp}})\\
&=&\frac{d^2}{dt\,ds}\Big|_{t,s=0}\A\left(\A\big(\exp^{tu}\big)\exp^{sv}
\A\big(\exp^{tu}\big)^{-1}\right)\hspace{1cm}(\text{by  ~\meqref{eq:pinv}})\\
&=&\A_{\ast e}\left(\frac{d^2}{dt\,ds}\Big|_{t,s=0}\A\Big(\exp^{tu}\Big)\exp^{sv}
\A\Big(\exp^{tu}\Big)^{-1}\right) \\
&=& \Aa \left(\frac{d^2}{dt\,ds}\Big|_{t,s=0} \exp^{tA(u)}\exp^{sv}
\exp^{-tA(u)} \right)\hspace{1cm}(\text{by  ~\meqref{eq:tang}})\\
&=&\Aa([\Aa(u),v]) \hspace{2cm}(\text{by  ~\meqref{eq:lie}}).
\end{eqnarray*}
Therefore $(\mathfrak{g}, \Aa)$ is an averaging Lie algebra.
\end{proof}

\subsection{Averaging Hopf algebras} This subsection is devoted to averaging Hopf algebras, in particular
the relationship between averaging groups and averaging Hopf algebras.

\begin{defi}
 Let $(H, \mu, \eta, \bigtriangleup, \iota, S)$ be a Hopf algebra.
 A coalgebra map $\Aa$ on $H$ is called an {\bf \aveo} if
\begin{equation}\mlabel{eq:avehopf}
\Aa(a)\Aa(b)=\Aa(\Aa(a)b)=\Aa(a\Aa(b))\tforall a,b\in H.
\end{equation}
In this case, the pair $(H, \Aa)$ is called an {\bf averaging  Hopf algebra}.
\end{defi}

Notice that the identity map $\id_H$ is an averaging operator on $H$.
The next result shows that each cocommutative Hopf algebra equipped with the idempotent antipode is an averaging Hopf algebra.

\begin{pro}
  Let $(H, \mu, \eta, \bigtriangleup, \iota, S)$ be a cocommutative Hopf algebra with $S^2=S$. Then $(H, S)$ is an  averaging  Hopf algebra.
\end{pro}
\begin{proof}
It is well known that $S$ is an algebra and coalgebra anti-homomorphism, that is,
$$S(ab)=S(b)S(a),\quad \Delta\circ S(a)=\sum_{(a)}S(a_2)\otimes S(a_1)\tforall a,b\in H.$$
Since $H$ is cocommutative,  $S$ is a coalgebra homomorphism.
Further by $S^2=S$, we have
$$S(a)S(b)=S(ba)=S(S(ba))=S(S(a)S(b))=S(S(b))S(S(a))=S(b)S(a).$$
Thus
$$S(a)S(b)=S(b)S(a)=S(S(b))S(a)=S(aS(b))$$
and
$$S(a)S(b)=S(b)S(a)=S(b)S(S(a))=S(S(a)b),$$
as required.
\end{proof}

Recall that for any group $G$, its group algebra $\bfk[G]$ is a cocommutative Hopf algebra with diagonal coproduct
\begin{equation*}\mlabel{eq:ahopf}
\Delta:\bfk[G]\to \bfk[G]^{\otimes2},\quad g\mapsto g\otimes g
\end{equation*}
and antipode
$$S: \bfk[G]\to \bfk[G],\quad g\mapsto g^{-1}.$$

\begin{thm}\mlabel{thm:hopf}
Let $G$ be a group and $\A:G\rightarrow G$ a map. Denote by
\begin{equation*}\mlabel{eq:lave}
\Aa: \bfk[G] \to \bfk[G],\quad \sum_i\alpha_ig_i\mapsto\sum_i\alpha_i\A(g_i)
\end{equation*}
the linear extendsion of $\A$. Then the pair $(G, \A)$ is an averaging group if and only if the pair $(\bfk[G], \Aa)$ is an averaging Hopf algebra.
\end{thm}

\begin{proof}
It follows from~\meqref{eq:avegp}, \meqref{eq:avehopf} and linearity.
\end{proof}

\begin{rmk}
It was shown~\mcite{Gon} that a Rota-Baxter operator on a Lie algebra can be extended to the Hopf algebra of its universal enveloping algebra. However, the method in~\mcite{Gon} fails for the case of averaging operators.
It is challenging to extend an averaging operator from a Lie algebra to the Hopf algebra of its universal enveloping algebra.
\end{rmk}

\section{Free averaging groups}
\mlabel{sec:rbgp}
In this section, we construct explicitly the free averaging group on a set. To be specific, we first give the following concept.

\begin{defi}
Let $X$ be a set. The \name{free averaging group} on $X$ is an averaging group $(\frakA(X), \A_X)$ together with a map $j_X:X\to \frakA(X)$ such that, for any averaging group $(G,\A)$ and map $\varphi:X\to G$, there is a unique morphism $\free{\varphi}:(\frakA(X),\A_X)\to (G,\A)$ of averaging groups such that $\free{\varphi}\circ j_X=\varphi$.
\end{defi}

A formal construction of free averaging groups obtained by taking a quotient as follows.
Let a set $X$ be given and let $(\calg(X),P_X)$ be the free operated group generated by $X$. Let $N_X$ be the normal operated subgroup of $\calg(X)$ generated by elements of the form
$$ P_X(u)P_X(v) \Big(P_X(u{P_X(v)})\Big)^{-1},\quad P_X(P_X(u)v) \Big(P_X(u{P_X(v)})\Big)^{-1}
\tforall u, v\in \calg(X).$$
Then the quotient group $\calg(X)/N_X$, with the operator inherited from $P_X$ modulo $N_X$, is automatically the free averaging group on $X$.

We next give a direct and explicit construction of the free averaging group on $X$.
It is known~\mcite{GGLZ} that the free Rota-Baxter group on a set is defined on the set of Rota-Baxter group words $\frakR(X) \subset \calg(X)$ consisting of words that do not contain a subword of the form $\lc u\rc\lc v\rc$ with $u, v\in \calg(X)$. This motivates us to consider the averaging
case in a similar way. It is natural to choose the subset of $\calg(X)$  that do not contain a subword of
the forms $\lc u\rc\lc v\rc$  and $\lc \lc u\rc v\rc$ with $u,v\in \calg(X)$.
However, this condition is not enough. Indeed, for an averaging group $(G, \A)$ and $u,v\in G$,
\begin{equation}\mlabel{eq:aa2}
\A(u  \A^2(v)) =   \A(u)   \A^2(v) =   \A(  \A(u)  \A(v)) =   \A^2(u  \A(v)),
\end{equation}
and repeatedly,
\begin{equation*}\mlabel{eq:aan}
\A(u  \A^n(v)) =    \A^n(u  \A(v)),
\end{equation*}
which, together with (\mref{eq:avegp}), implies that
\begin{align*}
 \A^s(u)\A^t(v) =&\ \A^s(u\A^t(v)) = \A^{s+t-1}(u\A(v))\\
 \A(\A^s(u)v)=&\ \A(u\A^s(v)) = \A^s(u\A(v)) \tforall u,v\in G, \, s,t\geq 1.
\end{align*}
Notice also that $$\Big\{ \lc u \lc v\rc^{(n)}\rc \mid u,v\in  \calg(X), n\geq 1 \Big\} \subseteq \Big\{ \lc u \lc v\rc^{(2)}\rc \mid u,v\in  \calg(X)\Big\},$$
since $\lc v\rc^{(n)}$ can be viewed $\lc \lc v\rc^{(n-2)} \rc^{(2)}$.
The equation (\mref{eq:aa2}) motives us to give the following concept.

\begin{defi}
An element $w\in\calg(X)$ is called an \name{\aveg word} on a set $X$ if $w$ does not contain any subword of the form $\lc u\rc\lc v\rc$,  $\lc \lc u\rc v\rc$ or $\lc u\lc v\rc^{(2)}\rc$ with $u, v\in \calg(X)$.
Denote by $\rbgw$ the set of all \aveg words on $X$.
\end{defi}

We are going to define a multiplication $$\diamond:\rbgw \times \rbgw \rightarrow \rbgw, \quad (u, v) \mapsto u\diamond v$$ by induction on $(\dep(u), \dep(v))$ as follows. For the initial step of $(\dep(u), \dep(v)) = (0,0)$, define $u\diamond v:= uv$ to be the concatenation of $u$ and $v$. Let $m,n\geq 0$ be given such that $(m,n)>(0,0)$ lexicographically. Suppose that the product $u\diamond v$ have been defined for $u, v\in \rbgw$ with $(\dep(u), \dep(v))<(m,n)$ lexicographically and consider $u, v\in \rbgw$ with $(\dep(u), \dep(v))=(m,n)$.
There are two cases to consider: $\bre(u)=\bre(v)=1$ and $\bre(u)+\bre(v)>2$.

\noindent{\bf Case 1.} If $\bre(u)=\bre(v)=1$, then define
\begin{equation}\mlabel{eq:bre1}
u\diamond v:=\left\{
  \begin{array}{ll}
    \lc u'\diamond\lc v' \rc\rc^{(s+t-1)}, & \hbox{if ~$u=\lc u' \rc^{(s)}, v=\lc v'\rc^{(t)}$, $u'\notin\lc\rbgw\rc$;} \\
     \big(\lc v'\diamond\lc u'\rc\rc^{(s+t-1)}\big)^{-1}, & \hbox{if ~$u=\big(\lc u' \rc^{(s)}\big)^{-1}, v=\big(\lc v' \rc^{(t)}\big)^{-1}$, $v'\notin\lc\rbgw\rc$;}\\
    uv, & \hbox{otherwise.}
  \end{array}
\right.
\end{equation}
\noindent
{\bf Case 2.}
If $\bre(u)+\bre(v)=k+\ell>2$, write $u=u_{1}\cdots u_{k}$ and $v=v_{1}\cdots v_{\ell}$ in the standard factorizations in  ~\meqref{eq:frbg}. Define
\begin{equation}
\mlabel{eq:stdia}
u\diamond v:=\left\{
  \begin{array}{ll}
    v, & \hbox{if ~$k=0$, that is, $u=1$;} \\
    u, & \hbox{if ~$\ell=0$, that is, $v=1$;} \\
     u_{1} \cdots u_{k-1} (u_{k}\diamond v_{1}) v_{2}\cdots v_{\ell}, & \hbox{otherwise.}
  \end{array}
\right.
\end{equation}
where $u_{k}\diamond v_{1}$ is defined by Case 1.

\begin{ex}
Let $u=\lc x \lc y\rc \rc^{(2)}, v=\lc z\rc^{-1}$ and $w=\lc z\rc^{(3)}$ with $x,y,z\in X$. Then
$$u\diamond v= \lc x \lc y\rc \rc^{(2)}\lc z\rc^{-1},~v\diamond w=\lc z\rc^{-1}\lc z\rc^{(3)},~ u\diamond w= \lc x \lc y\rc\diamond \lc z\rc \rc^{(4)}=\lc x \lc y\lc z\rc \rc \rc^{(4)}.$$
\end{ex}

Now we define an operator
$$\A_X:\rbgw\rightarrow \rbgw, \quad w\mapsto \A_X(w)$$
as follows.
For $w\in \rbgw\subseteq \calg(X)$, we first write $$w=w_{1}\cdots w_{k}\,\text{ with }\, w_1, \cdots, w_k\in \rbgw$$
in the standard factorization in  ~\meqref{eq:frbg}. Employing the induction on $\degp(w)\geq 0$, we then define
{\small \begin{equation}\label{eq:daveo}
\A_X(w):=\left\{
            \begin{array}{ll}
              \lc w\rc, & \hbox{if $k=0,1$, or $k\geq 2$ and $w_1, w_k\notin \lc\rbgw\rc$ ;} \\
\Big\lc w_1'\diamond \A_X(w_2\cdots w_k)\Big\rc^{(t)} , & \hbox{if $k\geq2$ with $w=\lc w_1'\rc^{(t)}w_2\cdots w_k$, $w_1'\notin\lc\rbgw\rc$;}\\
\Big\lc w_1\cdots w_{k-1}\A_X(w_k')\Big\rc^{(s)} ,& \hbox{if $k\geq2$ with $w=w_1\cdots w_{k-1}\lc w_k'\rc^{(s)}$, $w_1,w_k'\notin \lc\rbgw\rc$ .}
            \end{array}
          \right.
\end{equation}}
Here $s,t\geq 1$. According to whether $w_k$ is in $\lc \rbgw\rc$, the above second case can be divided  into two subcases.
So equivalently,
{\small \begin{equation*}\label{eq:daveo1}
\A_X(w):=\left\{
  \begin{array}{ll}
    \lc w \rc , & \hbox{if $k=0,1$, or $k\geq 2$ and $w_1, w_k\notin \lc\rbgw\rc$ ;}  \\
   \Big \lc w_1'\diamond \lc w_2\cdots w_k\rc\Big\rc^{(t)}, & \hbox{if $k\geq 2$ and $w_1=\lc w_1' \rc^{(t)}, w_1', w_k \notin \lc\rbgw\rc$;} \\
     \Big\lc w_1'\diamond \lc w_2\cdots \lc w_k'\rc\rc\Big\rc^{(s+t-1)}, & \hbox{if $k\geq 2$ and $w_1=\lc w_1' \rc^{(t)}, w_k =\lc w_k' \rc^{(s)}$, $w_1', w_k'\notin\lc\rbgw\rc$;} \\
 \Big\lc w_1\cdots w_{k-1} \A_X(w_k')\Big\rc^{(s)}, & \hbox{if $k\geq 2$ and  $w_k=\lc w_k'\rc^{(s)}$, $w_1, w_k' \notin \lc\rbgw\rc$.}
   \end{array}
\right.
\end{equation*}}
Here notice that in the second and third two cases, $w_1=\lc w_1' \rc^{(t) }\in \lc \rbgw\rc$ implies that $w_2\notin \lc \rbgw\rc$ by $w=w_1\cdots w_k\in \rbgw$ has no subwords of the form $\lc u\rc \lc v\rc$. So respectively,
$$\A_X(w_2\cdots w_k) = \lc w_2\cdots w_k\rc \in \rbgw,\quad \A_X(w_2\cdots w_k) = \lc w_2\cdots \lc w_k'\rc\rc \in \rbgw.$$

\begin{ex}
Let $u=\lc x \lc y\rc \rc^{(2)}, v=\lc z\rc^{-1}$ and $w=\lc x\rc^{(3)}y\lc z\rc^{(2)}$ with $x,y,z\in X$. Then
$$\A_X(u)=\lc x \lc y\rc \rc^{(3)}, \quad\A_X(v)=\lc \lc z\rc^{-1} \rc, \quad \A_X(w)=\Big\lc x  \big\lc y\lc z\rc\big\rc\Big\rc^{(4)}.$$
\end{ex}

We expose the following result as a preparation.

\begin{lem}\mlabel{clm:dep}
 Define
$$\rbgw_1:=\{w\in \rbgw\mid \bre(w)=1\}.$$ Then
\begin{equation}\mlabel{eq:dias1}
(u \diamond v)\diamond w= u \diamond (v \diamond w)\tforall u,v,w\in \rbgw_1.
\end{equation}
\end{lem}

\begin{proof}
The result can be accomplished by an induction on $(\dep(u) , \dep(v) , \dep(w))\geq (0,0,0)$ lexicographically.
If $(\dep(u), \dep(v), \dep(w)) = (0, 0, 0)$, then $u, v,w$ are in the free group $\calg_0=G(X)$ and  ~\meqref{eq:dias1} holds by the associativity of the free group $G(X)$.

For the inductive step, let $p, q, r \geq 0$ be given such that $(p, q, r) > (0, 0, 0)$.
Assume that  ~\eqref{eq:dias1} has been proved for any $u,v,w\in\rbgw_1$ with $(\dep(u), \dep(v), \dep(w)) < (p, q, r)$.
We next show that   ~\meqref{eq:dias1} holds, for $u,v,w\in\rbgw_1$ with $(\dep(u), \dep(v), \dep(w)) = (p, q, r)$.
There are five cases to consider:
$$\left\{
  \begin{array}{ll}
& \hbox{$u,v,w\in\lc \rbgw \rc$;} \\
& \hbox{$u,v,w\in\lc \rbgw \rc^{-1}$;} \\
& \hbox{$u,v\in\lc \rbgw \rc$, $w\in\lc \rbgw \rc^{-1}$;} \\
 & \hbox{$u\in\lc \rbgw \rc$, $v,w\in\lc \rbgw \rc^{-1}$;} \\
 & \hbox{all other cases.}
\end{array}
\right.$$

\noindent{\bf Case 1.} Suppose $u,v,w\in\lc \rbgw \rc$ with $u=\lc u' \rc^{(s)}, v=\lc v' \rc^{(t)}$ and $w=\lc w' \rc^{(r)}$.  By  ~\meqref{eq:bre1} and  the induction hypothesis,
\begin{eqnarray*}
(u \diamond v)\diamond w
&=& \Big\lc u'\diamond \lc v' \rc\Big\rc^{(s+t-1)} \diamond\lc w' \rc^{(r)}\\
&=& \Big\lc \big(u' \diamond\lc v' \rc\big) \diamond \lc w'\rc \Big\rc^{^{(s+t+r-2)}}\\
&=&\left\{
   \begin{array}{ll}
    \Big\lc\big( u'_1 \cdots u'_{k-1}\lc u_k'\rc \diamond\lc v' \rc \big)\diamond \lc w'\rc \Big\rc^{^{(s+t+r-2)}}, & \hbox{if $u'=u'_1 \cdots u'_{k-1}\lc u_k'\rc$;} \\
    \Big\lc u'_1\cdots  u'_k \lc v' \rc \diamond \lc w'\rc \Big\rc^{^{(s+t+r-2)}} , & \hbox{if $u'=u'_1\cdots  u'_k$, $u'_k\notin \lc \rbgw\rc$ .}
   \end{array}
 \right.\\
&=&\left\{
   \begin{array}{ll}
    \Big\lc u'_1 \cdots u'_{k-1}\lc u_k' \rc \diamond \big(\lc v'\rc \diamond \lc w'\rc \big) \Big\rc^{^{(s+t+r-2)}}, & \hbox{if $u'=u'_1 \cdots u'_{k-1}\lc u_k'\rc$;} \\
    \Big\lc u'_1\cdots  u'_k \lc v'  \diamond \lc w'\rc\rc \Big\rc^{^{(s+t+r-2)}} , & \hbox{if $u'=u'_1\cdots  u'_k$, $u'_k\notin \lc \rbgw\rc$ .}
   \end{array}
 \right.\\
&=&\left\{
   \begin{array}{ll}
    \Big\lc u'_1 \cdots u'_{k-1}\big\lc u_k' \diamond\lc v' \diamond \lc w'\rc\rc\big\rc  \Big\rc^{^{(s+t+r-2)}}, & \hbox{if $u'=u'_1 \cdots 'u_{k-1}\lc u_k'\rc$;} \\
    \Big\lc u'_1\cdots  u'_k \lc v'  \diamond \lc w'\rc\rc \Big\rc^{^{(s+t+r-2)}} , & \hbox{if $u'=u'_1\cdots  u'_k$, $u'_k\notin \lc \rbgw\rc$ .}
   \end{array}
 \right.\\
&=&\Big\lc u' \diamond\lc v'  \diamond \lc w'\rc\rc \Big\rc^{^{(s+t+r-2)}}\\
&=&\lc u'\rc^{(s)}\diamond \Big\lc v' \diamond\lc w' \rc\Big\rc^{(t+r-1)}\\
&=& u \diamond (v \diamond w).
\end{eqnarray*}

\noindent{\bf Case 2.} Suppose $u,v,w\in\lc \rbgw \rc^{-1}$ with $u=(\lc u' \rc^{(s)})^{-1}, v=(\lc v' \rc^{(t)})^{-1}$ and $w=(\lc w' \rc^{(r)})^{-1}$.
By  ~\meqref{eq:bre1} and  Case $1$
\begin{eqnarray*}
(u \diamond v)\diamond w
&=&\Big( \lc v' \rc^{(t)}\diamond\lc u' \rc^{(s)}\Big)^{-1} \diamond\Big(\lc w' \rc^{(r)}\Big)^{-1}
= \Big(\lc w' \rc^{(r)}\diamond \Big(  \lc v' \rc^{(t)}\diamond\lc u' \rc^{(s)}\Big)\Big)^{-1}\\
&=&\Big(\Big( \lc w' \rc^{(r)}\diamond  \lc v' \rc^{(t)}\Big)\diamond\lc u' \rc^{(s)}\Big)^{-1}
=\Big(\lc u'\rc^{(s)} \Big)^{-1}\diamond\Big(\lc w' \rc^{(r)}\diamond\lc v' \rc^{(t)}\Big)^{-1}
=u \diamond (v \diamond w).
\end{eqnarray*}

\noindent{\bf Case 3.} Suppose $u,v\in\lc \rbgw \rc $ and $w\in \lc \rbgw \rc^{-1}$ with $u=\lc u' \rc^{(s)}, v=\lc v' \rc^{(t)}$ and $w=(\lc w' \rc^{(r)})^{-1}$.  By~\meqref{eq:bre1} and~\meqref{eq:stdia},
\begin{eqnarray*}
(u \diamond v)\diamond w
= \Big( \lc u' \rc^{(s)}\diamond\lc v' \rc^{(t)} \Big) \Big(\lc w' \rc^{(r)}\Big)^{-1}
=\lc u' \rc^{(s)}\diamond\Big(\lc v' \rc^{(t)}\Big(\lc w' \rc^{(r)}\Big)^{-1}\Big)
=u \diamond (v \diamond w).
\end{eqnarray*}

\noindent{\bf Case 4.}  Suppose $u\in\lc \rbgw \rc$ and $v,w\in \lc \rbgw \rc^{-1}$ with $u=\lc u' \rc^{(s)}, v=(\lc v' \rc^{(t)})^{-1}$ and $w=(\lc w' \rc^{(r)})^{-1}$. It follows from~\meqref{eq:bre1} and~\meqref{eq:stdia} that
\begin{eqnarray*}
(u \diamond v)\diamond w
&=&\Big(\lc u' \rc^{(s)}(\lc v' \rc^{(t)})^{-1}\Big)\diamond \Big(\lc w' \rc^{(r)}\Big)^{-1}
=\lc u' \rc^{(s)} \bigg( \Big(\lc v' \rc^{(t)}\Big)^{-1}\diamond \Big(\lc w' \rc^{(r)}\Big)^{-1} \bigg)
=u \diamond (v \diamond w).
\end{eqnarray*}

\noindent{\bf Case 5.} For  $u,v,w\in\rbgw_1$ with all other cases, by  ~\meqref{eq:stdia}, we have
$$(u \diamond v)\diamond w=uvw= u \diamond (v \diamond w).$$
Therefore  ~\meqref{eq:dias1} holds.
\end{proof}

Before we proceed to the free \avegs, we prove the following result.

\begin{thm}
With the above definitions of the multiplication $\diamond$  and the operator $\A_X$,
\begin{enumerate}
  \item the pair $(\rbgw,\diamond)$ is a group; \mlabel{item:gp}
  \item the triple $(\rbgw,\diamond,\A_X)$ is an \aveg.\mlabel{item:avegp}
\end{enumerate}
\mlabel{thm:avegp}
\end{thm}

\begin{proof}
\meqref{item:gp} Notice that the identity element $1$ and $w^{-1}$ are averaging group words, for any $w\in \rbgw$.
It suffices to check that $\diamond$ is associative:
\begin{equation}\mlabel{eq:dias}
(u \diamond v)\diamond w= u \diamond (v \diamond w)\tforall u,v,w\in\rbgw.
\end{equation}
Notice that  ~\meqref{eq:dias} holds if $u=1$, $v=1$ or $w=1$. Suppose
$u,v,w\neq 1$ and use induction on $\bre(u)+ \bre(v) + \bre(w))\geq 3$.

For the initial step of $\bre(u)+ \bre(v) + \bre(w))=3$, we have
$\bre(u)=\bre(v)=\bre(w)=1$, and so  ~\meqref{eq:dias} follows from Lemma~\mref{clm:dep}.

For the inductive step, let $k\geq 3$ be given.
Assume that  ~\eqref{eq:dias} has been proved for any $u,v,w\in\rbgw\setminus\{1\}$ with $\bre(u) + \bre(v) + \bre(w)\leq k$.
We now consider the case when $\bre(u) + \bre(v) + \bre(w)=k+1$.
Then $k \geq 4$ and so at least one of $\bre(u)$, $\bre(v)$ or $\bre(w)$ is
greater than or equal two. We accordingly have the following three cases to verify.

\noindent{\bf Case 1.} Suppose $\bre(u)\geq 2$. Then we may write
$u=u_{1}\cdots u_{m}$ in \meqref{eq:frbg} with $m\geq 2$. By~\meqref{eq:stdia} and the induction hypothesis,
\begin{align*}
(u\diamond v)\diamond w=&\ \big((u_{1}\cdots u_{m}) \diamond v\big) \diamond w\\
=&\ \big((u_{1}\cdots u_{m-1})(u_{m}\diamond v)\big)\diamond w   \\
=&\ (u_{1}\cdots u_{m-1})\big((u_{m}\diamond v)\diamond w\big) \\
=&\ (u_{1}\cdots u_{m-1})\big(u_{m}\diamond (v\diamond w)\big) \\
=&\ (u_{1}\cdots u_{m-1}u_{m})\diamond (v \diamond w) \\
=&\ u\diamond (v\diamond w).
\end{align*}

\noindent{\bf Case 2.}  Suppose $\bre(w)\geq 2$. Then the verification is similar to the previous case.

\noindent{\bf Case 3.} Suppose $\bre(v)\geq 2$. Write $v=v_{1}\cdots v_{n}$ in \meqref{eq:frbg} with $n\geq 2$. Repeatedly applying  ~\meqref{eq:stdia},
\begin{align*}
(u\diamond v)\diamond w=&\ \big(u\diamond (v_{1}\cdots v_{n})\big)\diamond w\\
=&\ \big( (u\diamond v_{1}) v_{2}\cdots v_{n}\big)\diamond w\\
=&\ (u\diamond v_{1})(v_{2}\cdots v_{n-1})(v_{n}\diamond w)\\
=&\ u\diamond \big( (v_{1}\cdots v_{n-1})(v_{n}\diamond w) \big) \\
=&\ u\diamond \big((v_{1}\cdots v_{n})\diamond w\big)\\
=&\ u\diamond (v\diamond w).
\end{align*}

Now we have completed the inductive proof of~\meqref{eq:dias}.

\noindent\meqref{item:avegp}
We only need to verify the equations
\begin{equation}\mlabel{eq:fave}
\A_X(u)\diamond \A_X(v)=\A_X(u\diamond \A_X(v)),\quad \A_X(u)\diamond \A_X(v)=\A_X(\A_X(u)\diamond v)\tforall u,v\in \rbgw,
\end{equation}
which will be achieved by an induction on the sum of degrees $\degp(u)+ \degp(v) \geq 0$.
If $\degp(u)+\degp(v) = 0$, then
\begin{align*}
 \A_X(u) \diamond \A_X(v) \overset{\eqref{eq:daveo}}{=}& \lc u\rc \diamond \lc v\rc
 \overset{\eqref{eq:bre1}}{=}  \lc u\lc v\rc\rc = \A_X(u\diamond \A_X(v)),\\
 \A_X(\A_X(u) \diamond v )\overset{\eqref{eq:daveo}}{=}&  \A_X(\lc u\rc v) \overset{\eqref{eq:daveo}}{=}\lc u\diamond\A_X( v)\rc =\lc u\lc v\rc\rc\overset{\eqref{eq:daveo}}{=}  \A_X(u) \diamond \A_X(v).
\end{align*}
For the inductive step, let $\ell\geq 0$ be given.
Assume that  ~\eqref{eq:fave} has been proved for any $u,v\in\rbgw$ with $\degp(u) + \degp(v) \leq \ell$.
We now consider the case when $\degp(u) + \degp(v) =\ell+1$, which will be proved by an induction on the sum of breadths $\bre(u)+ \bre(v) \geq 2$ with $u,v\neq1$, since the result holds for $u=1$ or $v=1$.
Indeed, if $u=1$ or $v=1$, without loss of generality, let $u=1$. Then $\A_X(v) =  \lc v'\rc $ for some $v'\in \A(X)$ by~\meqref{eq:daveo} and so
\begin{align*}
 \A_X(1) \diamond \A_X(v) \overset{\eqref{eq:daveo}}{=}& \lc 1\rc \diamond \lc v'\rc
 \overset{\eqref{eq:bre1}}{=}  \lc\lc v'\rc\rc = \lc \A_X(v)\rc \overset{\eqref{eq:daveo}}{=} \A_X(\A_X(v)) = \A_X(1\diamond \A_X(v)),\\
 \A_X(1) \diamond \A_X(v) \overset{\eqref{eq:daveo}}{=}&  \lc\lc v'\rc\rc =\lc1\diamond\A_X( v)\rc \overset{\eqref{eq:daveo}}{=}\A_X (\lc 1\rc \diamond v)\overset{\eqref{eq:daveo}}{=} \A_X( \A_X(1) \diamond v).
\end{align*}

\noindent{\textbf{I. The initial step on $\bre(u)+ \bre(v)$.}}
For the initial step $\bre(u)+ \bre(v)=2$, we have $\bre(u)= \bre(v)=1.$
According to whether $u$ or $v$ is in $\lc \rbgw \rc$,
we have the following four cases to consider.

\noindent{\bf Case 1.} $u,v\notin \lc \rbgw \rc$.
Then
\begin{align*}
\A_X(u)\diamond \A_X(v)\overset{\eqref{eq:daveo}}{=}\lc u \rc \diamond \lc v\rc
 \overset{\eqref{eq:bre1}}{=} \lc u \diamond \lc v\rc\rc= \lc u  \lc v\rc\rc
 \overset{\eqref{eq:daveo}}{=} \A_X(u\lc v\rc)
=\A_X(u\diamond \lc v\rc)
=\A_X(u\diamond \A_X(v))
\end{align*}
and
\begin{align*}
\A_X(u)\diamond \A_X(v)= \lc u \diamond \lc v\rc\rc  =\lc u \diamond \A_X( v)\rc
 \overset{\eqref{eq:daveo}}{=} \A_X( \lc u\rc   v)
 \overset{\eqref{eq:bre1}}{=} \A_X( \lc u\rc \diamond  v)
\overset{\eqref{eq:daveo}}{=} \A_X(\A_X(u)\diamond v).
\end{align*}

\noindent{\bf Case 2.} $u\in \lc \rbgw \rc$ and $v\notin \lc \rbgw \rc$.
Write $u=\lc u '\rc^{(s)}$ with $u' \notin \lc\rbgw\rc$. Then
\begin{eqnarray*}
\A_X(u)\diamond \A_X(v)&=&\lc u' \rc^{(s+1)} \diamond \lc v\rc =\Big\lc u' \diamond \lc v\rc\Big\rc^{(s+1)}  \overset{\eqref{eq:daveo}}{=}\A_X\Big(\Big\lc u' \diamond \lc v\rc\Big\rc^{(s)}\Big)\\
&\overset{\eqref{eq:bre1}}{=}& \A_X\Big( \lc u'\rc^{(s)} \diamond \lc v\rc\Big)=\A_X(u\diamond \A_X(v))
\end{eqnarray*}
and
\begin{eqnarray*}
\A_X(u)\diamond \A_X(v)&=&\Big\lc u' \diamond \lc v\rc\Big\rc^{(s+1)} = \Big\lc u' \diamond \A(v)\Big\rc^{(s+1)}    \overset{\eqref{eq:daveo}}{=}\A_X\Big( \lc u'\rc^{(s+1)}   v\Big)\\
&=&\A_X\Big( \lc u'\rc^{(s+1)} \diamond  v\Big)=\A_X(\A_X(u)\diamond v).
\end{eqnarray*}

\noindent{\bf Case 3.} $u\notin \lc \rbgw \rc$ and $v\in \lc \rbgw \rc$.
The verification of this case is similar to the previous case.

\noindent{\bf Case 4.} $u,v\in \lc \rbgw \rc$.
Suppose $u=\lc u '\rc^{(s)}, v=\lc v '\rc^{(t)}\in \lc \rbgw \rc$ with $u', v' \notin \lc\rbgw\rc$. Then
\begin{eqnarray*}
\A_X(u)\diamond \A_X(v)&=&\lc u' \rc^{(s+1)} \diamond \lc v '\rc^{(t+1)} \overset{\eqref{eq:bre1}}{=}\Big\lc u' \diamond \lc v'\rc\Big\rc^{(s+t+1)}  \overset{\eqref{eq:daveo}}{=}\A_X\Big(\Big\lc u' \diamond \lc v'\rc\Big\rc^{(s+t)}\Big)\\
&\overset{\eqref{eq:bre1}}{=}&\A_X\Big( \lc u'\rc^{(s)} \diamond \lc v'\rc^{(t+1)}\Big)=\A_X(u\diamond \A_X(v))
\end{eqnarray*}
and
\begin{eqnarray*}
\A_X(u)\diamond \A_X(v)=\A_X\Big(\Big\lc u' \diamond \lc v'\rc\Big\rc^{(s+t)}\Big)
\overset{\eqref{eq:bre1}}{=}\A_X\Big( \lc u'\rc^{(s+1)} \diamond \lc v'\rc^{(t)}\Big)=\A_X(\A_X(u)\diamond v).
\end{eqnarray*}

\noindent{\textbf{II. The inductive step on $\bre(u)+ \bre(v)$.}}
For the inductive step on $\bre(u) + \bre(v)\geq 3 $, let $k\geq 2$ be fixed.
Assume that  ~\eqref{eq:fave} has been showed for any $u,v\in\rbgw\setminus\{1\}$ with $\bre(u) + \bre(v)\leq k$.
We now consider the case when $\bre(u) + \bre(v)=k+1$.
Then at least one of $\bre(u)=:m$ or $\bre(v)=:n$ is greater than one.
We write $u=u_1\cdots u_m$ and $v=v_1\cdots v_n$ in \meqref{eq:frbg}.

We just to prove the first equation in~\meqref{eq:fave}, as the proof of the second equation in~\meqref{eq:fave} is similar to the former.
By~\meqref{eq:daveo}, there are three cases to consider for $u=u_1\cdots u_m$.

\noindent{\bf Case 1.} Suppose $m=0,1$, or $m>1$ and $u=u_1\cdots u_m$ with $u_1, u_m\notin \lc\rbgw\rc$. Then $\A_X(u)=\lc u\rc $. According to~\meqref{eq:daveo}, there are three subcases for $v=v_1\cdots v_n$.

\noindent{\bf Subcase 1.1.} Suppose $n=0,1$, or $n>1$ and $v=v_1\cdots v_n$ with $v_1, v_n\notin \lc\rbgw\rc$. Then
$$\A_X(u)\diamond \A_X(v)\overset{\eqref{eq:daveo}}{=}   \lc u\rc \diamond \lc v\rc\overset{\eqref{eq:bre1}}{=}\lc u \diamond \lc v\rc\rc\overset{\eqref{eq:bre1}, \eqref{eq:stdia}}{=}\lc u  \lc v\rc\rc
\overset{\eqref{eq:daveo}}{=}\A_X( u  \lc v\rc)\overset{\eqref{eq:daveo}}{=}\A_X(u\A_X(v))\overset{\eqref{eq:bre1}, \eqref{eq:stdia}}{=}\A_X(u\diamond \A_X(v)).$$

\noindent{\bf Subcase 1.2.} Suppose $n\geq2$ with $v=\lc v_1'\rc^{(p)}v_2\cdots v_n$ and $v_1' \notin \lc\rbgw\rc$. Then
\begin{eqnarray*}
\A_X(u)\diamond \A_X(v)&=& \lc u \rc \diamond \Big\lc v_1'\diamond \A_X(v_2\cdots v_n)\Big\rc^{(p)} \hspace{1cm}(\text{by  ~\meqref{eq:daveo}})\\
&=&\Big\lc u \diamond\lc v_1'\diamond \A_X(v_2\cdots v_n)\rc\Big\rc^{(p)}\hspace{1cm}(\text{by~\meqref{eq:bre1}})\\
&=&\Big\lc u \lc v_1'\diamond \A_X(v_2\cdots v_n)\rc\Big\rc^{(p)} \hspace{1cm}(\text{by $u_m\notin \lc \A(X)\rc$}) \\
&=&\Big\lc u \A_X\Big( v_1'\diamond \A_X(v_2\cdots v_n)\Big)\Big\rc^{(p)} \hspace{1cm}(\text{by  ~\meqref{eq:daveo}})\\
&=&\A_X\Big( u\lc v_1'\diamond \A_X(v_2\cdots v_n)\rc^{(p)}\Big)\hspace{1cm}(\text{by ~\eqref{eq:daveo}})\\
&=& \A_X\Big( u\diamond\lc v_1'\diamond \A_X(v_2\cdots v_n)\rc^{(p)}\Big) \hspace{1cm}(\text{by $u_m\notin \lc \A(X)\rc$})\\
&=&\A_X(u\diamond \A_X(v))  \hspace{1cm}(\text{by  ~\meqref{eq:daveo}}).
\end{eqnarray*}

\noindent{\bf Subcase 1.3.} Suppose $n\geq2$ with $v=v_1\cdots v_{n-1}\lc v_n'\rc^{(q)}$ and $v_1, v_n'\notin \lc\rbgw\rc$. We have
\begin{eqnarray*}
\A_X(u)\diamond \A_X(v)&=&\lc u\rc \diamond \Big\lc v_1\cdots v_{n-1} \A_X( v_n')\Big\rc^{(q)} \hspace{1cm}(\text{by  ~\meqref{eq:daveo}})\\
&=&\Big\lc u\diamond \lc v_1\cdots v_{n-1} \A_X( v_n')\rc \Big\rc^{(q)}\hspace{1cm}(\text{by  ~\meqref{eq:bre1}})\\
&=&\Big\lc u\lc v_1\cdots v_{n-1} \A_X( v_n')\rc \Big\rc^{(q)}  \hspace{1cm}(\text{by $u_m\notin \lc \A(X)\rc$})\\
&=&\Big\lc u \A_X (v_1\cdots v_{n-1} \A_X( v_n')) \Big\rc^{(q)} \hspace{1cm}(\text{by  ~\meqref{eq:daveo}}) \\
&=&\A_X\Big( u\lc v_1\cdots v_{n-1} \A_X( v_n')\rc^{(q)} \Big)\hspace{1cm}(\text{by  ~\meqref{eq:daveo}})\\
&=&\A_X\Big( u\diamond\lc v_1\cdots v_{n-1} \A_X( v_n')\rc^{(q)} \Big)  \hspace{1cm}(\text{by $u_m\notin \lc \A(X)\rc$})\\
&=&\A_X(u\diamond \A_X(v)).  \hspace{1cm}(\text{by  ~\meqref{eq:daveo}})
\end{eqnarray*}
\noindent{\bf Case 2.} Suppose $m\geq2$ with $u=\lc u_1'\rc^{(t)}u_2\cdots u_m$, $u_1' \notin \lc\rbgw\rc$. By  ~\meqref{eq:daveo}, there are three subcases for $v=v_1\cdots v_n$.

\noindent{\bf Subcase 2.1.} Suppose $n=0,1$, or $n>1$ and $v=v_1\cdots v_n$ with $v_1, v_n\notin \lc\rbgw\rc$. Then $\A_X(v)=\lc v\rc$ and
\begin{eqnarray*}
\A_X(u)\diamond \A_X(v)&=& \Big\lc u_1'\diamond \A_X(u_2\cdots u_m)\Big\rc^{(t)} \diamond \lc v\rc  \hspace{1cm}(\text{by  ~\meqref{eq:daveo}})\\
&=&\Big\lc u_1'\diamond \A_X(u_2\cdots u_m) \diamond \lc v\rc \Big\rc^{(t)}\hspace{1cm}(\text{by  ~\meqref{eq:bre1}})\\
&=&\Big\lc u_1'\diamond \A_X(u_2\cdots u_m) \diamond \A_X( v) \Big\rc^{(t)}\\
&=&\Big\lc u_1'\diamond \A_X(u_2\cdots u_m \diamond \A_X( v))  \Big\rc^{(t)}\hspace{1cm}(\text{by the induction hypothesis on breadth})\\
&=&\A_X\Big( \lc u_1'\rc^{(t)}u_2\cdots u_m \diamond \A_X( v)  \Big)\hspace{1cm}(\text{by  ~\meqref{eq:daveo}})\\
&=&\A_X(u\diamond \A_X(v)).
\end{eqnarray*}
\noindent{\bf Subcase 2.2.} Suppose $n\geq2$ with $v=\lc v_1'\rc^{(p)}v_2\cdots v_n$ and $v_1' \notin \lc\rbgw\rc$. We have
\begin{eqnarray*}
\A_X(u)\diamond \A_X(v)&=&\Big\lc u_1'\diamond \A_X(u_2\cdots u_{m})\Big\rc^{(t)}\diamond \Big\lc v_1'\diamond \A_X(v_2\cdots v_{n})\Big\rc^{(p)}\hspace{1cm}(\text{by  ~\meqref{eq:daveo}})\\
%
&=&\Big\lc u_1'\diamond \A_X(u_2\cdots u_{m})\diamond\Big\lc v_1'\diamond \A_X(v_2\cdots v_{n})\Big\rc\Big\rc^{(t+p-1)}\hspace{0.7cm}(\text{by~\eqref{eq:bre1} and $u_1'\notin \lc \A(X)\rc$})\\
&=&\Big\lc u_1'\diamond \A_X(u_2\cdots u_{m})\diamond \A_X( \lc v_1'\rc v_2\cdots v_{n})\Big\rc^{(t+p-1)}\hspace{1cm}(\text{by~\eqref{eq:daveo}})\\
&=&  \Big\lc u_1' \diamond\A_X\Big(u_2\cdots u_{m}\diamond\A_X( \lc v_1'\rc v_2\cdots v_{n})\Big)\Big\rc^{(t+p-1)}\\
&&\hspace{2cm}(\text{by the induction hypothesis on degree})\\
&=&\Big\lc\A_X\Big( \lc u_1'\rc^{(t)} u_2\cdots u_{m}\diamond\A_X( \lc v_1'\rc v_2\cdots v_{n})\Big)\Big\rc^{(p-1)}\hspace{1cm}(\text{by  ~\meqref{eq:daveo}})\\
&=&\A_X^{p}\Big(u\diamond \A_X\big( \lc v_1'\rc v_2\cdots v_{n}\big)\Big)\hspace{1cm}(\text{by  ~\meqref{eq:daveo}})\\
&=&\A_X\Big(u\diamond \A_X^{p}\big( \lc v_1'\rc v_2\cdots v_{n}\big)\Big)\hspace{0.7cm}(\text{by the induction hypothesis on degree})\\
&=&\A_X\Big(u\diamond \big\lc v_1' \A_X(v_2\cdots v_{n})\big\rc^{(p)} \Big)\hspace{0.7cm}(\text{by  ~\meqref{eq:daveo}})\\
&=&\A_X\Big(u\diamond \A_X\big( \lc v_1'\rc^{(p)} v_2\cdots v_{n}\big)\Big)\hspace{0.7cm}(\text{by  ~\meqref{eq:daveo}})\\
&=&\A_X(u\diamond \A_X(v)).
\end{eqnarray*}

\noindent{\bf Subcase 2.3.} Suppose $n\geq2$ with $v=v_1\cdots v_{n-1}\lc v_n'\rc^{(q)}$ and $v_1, v_n'\notin \lc\rbgw\rc$. We obtain
\begin{eqnarray*}
\A_X(u)\diamond \A_X(v)&=&\Big\lc u_1'\diamond \A_X(u_2\cdots u_{m})\Big\rc^{(t)} \diamond \Big\lc v_1\cdots v_{n-1} \A_X( v_{n}')\Big\rc^{(q)} \hspace{1cm}(\text{by  ~\meqref{eq:daveo}})\\
&=&\Big\lc u_1'\diamond\A_X(u_2\cdots u_{m})\diamond\Big\lc v_1\cdots v_{n-1} \A_X( v_{n}')\Big\rc\Big\rc^{(t+q-1)}\hspace{0.7cm}(\text{by  ~\meqref{eq:bre1} and $u_1'\notin \lc \A(X)\rc$})\\
&=&\Big\lc u_1'\diamond\A_X(u_2\cdots u_{m})\diamond \A_X\Big( v_1\cdots v_{n-1} \lc v_{n}'\rc\Big)\Big\rc^{(t+q-1)}\hspace{0.7cm}(\text{by  ~\meqref{eq:daveo}})\\
&=&\Big\lc u_1'\diamond\A_X\Big(u_2\cdots u_{m} \diamond \A_X\big( v_1\cdots v_{n-1} \lc v_{n}'\rc\big)\Big)\Big\rc^{(t+q-1)}\\
&&\hspace{2cm}(\text{by the induction hypothesis on breadth})\\
&=&\Big\lc \A_X\Big(\lc u_1' \rc^{(t)}u_2\cdots u_{m} \diamond \A_X\big( v_1\cdots v_{n-1} \lc v_{n}'\rc\big)\Big)\Big\rc^{(q-1)}\hspace{0.7cm}(\text{by  ~\meqref{eq:daveo}})\\
&=& \A_X^{q}\Big(u \diamond \A_X\big( v_1\cdots v_{n-1} \lc v_{n}'\rc\big)\Big)\hspace{2cm}(\text{by  ~\meqref{eq:daveo}})\\
&=&  \A_X\Big(u \diamond \A_X^{q}\big( v_1\cdots v_{n-1} \lc v_{n}'\rc\big)\Big)\hspace{0.7cm}(\text{by the induction hypothesis on degree})\\
&=&\A_X\Big(u\diamond \Big\lc v_1\cdots v_{n-1}\A_X( v_{n}')\Big\rc^{(q)}\Big)\hspace{2cm}(\text{by  ~\meqref{eq:daveo}})\\
&=&\A_X(u\diamond \A_X(v_1\cdots v_{n-1}\lc v_{n}'\rc^{(q)}))\hspace{2cm}(\text{by  ~\meqref{eq:daveo}})\\
&=&\A_X(u\diamond \A_X(v)).
\end{eqnarray*}
\noindent{\bf Case 3.} Suppose $m\geq2$ with $u=u_1\cdots u_{m-1}\lc u_m'\rc^{(s)}$, $u_1, u_m'\notin \lc\rbgw\rc$. By  ~\meqref{eq:daveo}, there are three subcases for $v=v_1\cdots v_n$.

\noindent{\bf Subcase 3.1.} Suppose $n=0,1$, or $n>1$ and $v=v_1\cdots v_n$ with $v_1, v_n\notin \lc\rbgw\rc$. Then $\A_X(v)=\lc v\rc$ and
\begin{eqnarray*}
\A_X(u)\diamond \A_X(v)&=&\Big\lc u_1\cdots u_{m-1} \A_X( u_{m}')\Big\rc^{(s)} \diamond \lc v\rc \\
&=&\Big\lc u_1\cdots u_{m-1} \A_X( u_{m}') \diamond \lc v\rc \Big\rc^{(s)}\hspace{1.3cm}(\text{by  ~\meqref{eq:bre1}})\\
&=&\Big\lc u_1\cdots u_{m-1} \A_X( u_{m}') \diamond \A_X( v) \Big\rc^{(s)}\\
&=& \Big\lc u_1 \cdots u_{m-1}\A_X(u_{m}' \diamond \A_X(v))  \Big\rc^{(s)}
\hspace{0.5cm}(\text{by the induction hypothesis on degree})\\
&=&\A_X\Big( u_1 \cdots u_{m-1}\lc u_{m}' \diamond \A_X(v)\rc^{(s)}  \Big)\hspace{0.7cm}(\text{by  ~\meqref{eq:daveo}})\\
&=&\A_X\Big( u_1 \cdots u_{m-1}\A_X^s( u_{m}' )\diamond \A_X(v) \Big)\hspace{0.6cm}(\text{by the induction hypothesis on degree})\\
&=&\A_X\Big( u_1 \cdots u_{m-1}\lc u_{m}' \rc^{(s)}\diamond \A_X(v) \Big)\hspace{0.7cm}(\text{by  ~\meqref{eq:daveo} and $\lc u_m'\rc\in\lc\rbgw\rc$})\\
&=&\A_X(u\diamond \A_X(v)).
\end{eqnarray*}

\noindent{\bf Subcase 3.2.} Suppose $n\geq2$ with $v=\lc v_1'\rc^{(p)}v_2\cdots v_n$ and $v_1' \notin \lc\rbgw\rc$. Thus
\begin{eqnarray*}
\A_X(u)\diamond \A_X(v)&=& \Big\lc u_1\cdots u_{m-1} \A_X( u_{m}')\Big\rc^{(s)}\diamond \Big\lc v_1'\diamond \A_X(v_2\cdots v_{n})\Big\rc^{(p)} \\
&=&\Big\lc u_1\cdots u_{m-1} \A_X( u_{m}')\diamond\Big\lc v_1'\diamond \A_X(v_2\cdots v_{n})\Big\rc\Big\rc^{(s+p-1)}\hspace{1cm}(\text{by  ~\meqref{eq:bre1}})\\
&=&\Big\lc u_1\cdots u_{m-1} \A_X( u_{m}')\diamond\A_X( \lc v_1'\rc v_2\cdots v_{n})\Big\rc^{(s+p-1)}\hspace{1.3cm}(\text{by  ~\meqref{eq:daveo}})\\
&=& \Big\lc u_1\cdots u_{m-1} \A_X\Big(  u_{m}'\diamond \A_X( \lc v_1'\rc v_2\cdots v_{n})\Big)\Big\rc^{(s+p-1)} \\
&&\hspace{2cm}(\text{by the induction hypothesis on degree})\\
&=&\A_X\Big( u_1\cdots u_{m-1} \Big\lc  u_{m}'\diamond \A_X( \lc v_1'\rc v_2\cdots v_{n})\Big\rc^{(s+p-1)}\Big) \hspace{1.3cm}(\text{by  ~\meqref{eq:daveo}})\\
&=&\A_X\Big( u_1\cdots u_{m-1} \A_X^{s+p-1}\Big(  u_{m}'\diamond \A_X(  v_1'\diamond\A_X( v_2\cdots v_{n}))\Big)\Big) \hspace{1.3cm}(\text{by  ~\meqref{eq:daveo}})\\
&=&\A_X\Big( u_1\cdots u_{m-1} \A_X^s( u_{m}' )\diamond\A_X^p \big(v_1'\diamond \A_X(v_2\cdots v_{n})\big)\Big) \\
&&\hspace{2cm}(\text{by the induction hypothesis on degree})\\
&=&\A_X\Big( u \diamond\A_X(\lc v_1'\rc^{(p)}v_2\cdots v_{n})\Big) \hspace{2cm}(\text{by  ~\meqref{eq:daveo} and $\lc u_m'\rc\in\lc\rbgw\rc$})\\
&=&\A_X(u\diamond \A_X(v)).
\end{eqnarray*}
\noindent{\bf Subcase 3.3.} Suppose $n\geq2$ with $v=v_1\cdots v_{n-1}\lc v_n'\rc^{(q)}$ and $v_1, v_n'\notin \lc\rbgw\rc$. Similar to the subcase 3.2,
\begin{eqnarray*}
\A_X(u)\diamond \A_X(v)&=&\Big\lc u_1\cdots u_{m-1} \A_X( u_m')\Big\rc^{(s)}\diamond\Big\lc v_1\cdots v_{n-1} \A_X( v_n')\Big\rc^{(q)}\\
&=&\Big\lc u_1\cdots u_{m-1} \A_X( u_m')\diamond\Big\lc v_1\cdots v_{n-1} \A_X( v_n')\Big\rc\Big\rc^{(s+q-1)}\hspace{0.7cm}(\text{by  ~\meqref{eq:bre1}})
\end{eqnarray*}
is equal to $\A_X(u\diamond \A_X(v))$ .
Thus two equations in  ~\meqref{eq:fave} are verified.
This completes the proof.
\end{proof}

Having introduced all auxiliary results, we are ready for our main result in this section.

\begin{thm}\mlabel{thm:freea}
Let $X $ be a set. The triple $(\rbgw, \diamond, \A_X)$, together with the inclusion $j_X : X  \to \rbgw$, is the
free \aveg  on $X$.
\end{thm}

\begin{proof}
For any  \aveg $(G,\A)$ and any map $\varphi:X\to G$, it suffices to show that there is a unique \aveg homomorphism $$\lbar{\varphi}: \rbgw\to G, \quad w\mapsto \lbar{\varphi}(w)$$ such that $\varphi=\lbar{\varphi} \circ j_{X}$.

({\bf Existence})  We achieve an \aveg morphism $\lbar{\varphi} $ by applying the induction on $\dep(w)\geq 0$.
For the initial step of $\dep(w) = 0$, there are two cases to consider: $\bre(w)\leq 1$ or $\bre(w)\geq 2$.
If $\bre(w)\leq 1$, then $w=1$ or $w\in X\sqcup X^{-1}$. Then define
\begin{align}
\lbar{\varphi}(w):=
\begin{cases}
1_G, & \text{ if } w = 1,\\
\varphi(w),& \text{ if } w \in X,\\
\varphi(x)^{-1}, &\text{ if } w=x^{-1} \in X^{-1} \text{ for some } x\in X.
\end{cases}
\mlabel{eq:bfdep0}
\end{align}
If $\bre(w) \geq 2$, write $w=w_{1}\cdots w_{n}$ in the standard  factorization with $n\geq 2$ and define
\begin{equation}
\lbar{\varphi}(w) :=\lbar{\varphi}(w_{1}\cdots w_{n}) :=\lbar{\varphi}(w_{1})\cdots \lbar{\varphi}(w_{n}),
\mlabel{eq:rbdef3'}
\end{equation}
where each $\lbar{\varphi}(w_i)$ is defined in  ~(\mref{eq:bfdep0}).

For a fixed $k\geq 0$, assume that $\lbar{\varphi}(w)$ have been defined for $w\in \rbgw$ with $\dep(w)\leq k$ and consider the inductive step $w\in \rbgw$ with $\dep(w)=k+1$. There are two cases to consider: $\bre(w)= 1$ or $\bre(w)\geq 2$.

First consider $w$ with $\bre(w)= 1$. From $\dep(w)=k+1\geq 1$, there are two cases $w\in \lc \rbgw\rc$ or $w\in \lc \rbgw\rc^{-1}$. In the first case, write $w=\lc w'\rc=\A_X( {w'})$ for some ${w'}\in \rbgw$ with $\dep({w'})=k$. So the induction hypothesis allows us to define
\begin{equation}
\lbar{\varphi}(w):=\lbar{\varphi}\big(\lc {w'}\rc\big)=\lbar{\varphi}\big(\A_X( {w'})\big):= \A\big(\lbar{\varphi}({w'})\big).
\mlabel{eq:rbdef1}
\end{equation}
In the second case, write $w =\lc w'\rc^{-1}=\A_X( {w'})^{-1}$ for some ${w'}\in \rbgw$ with $\dep(w')=k$. Then the inductive hypothesis allows us to define
\begin{equation}
\lbar{\varphi}(w) :=\lbar{\varphi}\big(\lc {w'}\rc^{-1}\big)=\lbar{\varphi}\big(\A_X( {w'})^{-1}\big):=\A\big(\lbar{\varphi}({w'})\big)^{-1}.
\mlabel{eq:rbdef2}
\end{equation}

Next consider $w$ with $\bre(w)\geq 2$. Write $w=w_{1}\cdots w_{n}$ in the standard factorization in~\meqref{eq:frbg} with $n\geq 2$ and define
\begin{equation}
\lbar{\varphi}(w) :=\lbar{\varphi}(w_{1}\cdots w_{n}) :=\lbar{\varphi}(w_{1})\cdots \lbar{\varphi}(w_{n}),
\mlabel{eq:rbdef3}
\end{equation}
where each $\lbar{\varphi}(w_i)$ is defined in   ~(\mref{eq:bfdep0}), (\mref{eq:rbdef1}) or~(\mref{eq:rbdef2}).

By  ~(\ref{eq:rbdef1}), we have $\lbar{\varphi}  \circ \A_X = \A\circ \lbar{\varphi}$. So we are left to prove that $\lbar{\varphi}$ is a group homomorphism:
\begin{equation}
\lbar{\varphi}(u\diamond v)=\lbar{\varphi}(u)\lbar{\varphi}(v)\tforall~ u,v\in \rbgw,
\mlabel{eq:rbhomo}
\end{equation}
which will be achieved by an induction on $(\dep(u), \dep(v))\geq (0,0)$ lexicographically.

If $(\dep(u), \dep(v))=(0, 0)$, then $u, v\in \calg_{0}$ and  ~\eqref{eq:rbhomo} follows from  ~\eqref{eq:rbdef3'}.
Next fix $p,q\geq 0$ with $(p,q)>(0,0)$ lexicographically.
Assume that~\eqref{eq:rbhomo} holds for $u,v\in \rbgw$ with $(\dep(u), \dep(v))< (p, q)$
lexicographically, and consider $u,v\in \rbgw$ with $(\dep(u), \dep(v)) = (p, q)$.
There are four cases to consider:
$$\left\{
  \begin{array}{ll}
    & \hbox{  $\bre(u)=\bre(v)=1$ and $u=\lc u' \rc^{(s)}, v=\lc v' \rc^{(t)}$, $u, v \notin \lc\rbgw\rc$;} \\
   & \hbox{  $\bre(u)=\bre(v)=1$ and  $u=\big(\lc u' \rc^{(p)}\big)^{-1}, v=\big(\lc v' \rc^{(q)}\big)^{-1}$, $u, v \notin \lc\rbgw\rc$;} \\
& \hbox{ otherwise for $\bre(u)=\bre(v)=1;$} \\
 & \hbox{  $\bre(u)+\bre(v)\geq3$ .} \\
  \end{array}
\right.$$

\noindent{\bf Case 1.} Suppose that  $\bre(u)=\bre(v)=1$, $u=\lc u'\rc^{(s)}$ and $v=\lc v'\rc^{(t)}$ for some $u',v'\in \rbgw$.
By the induction hypothesis,
\begin{eqnarray*}
 \lbar{\varphi}(u\diamond v)
&=& \lbar{\varphi}\big(\lc u' \rc^{(s)} \diamond \lc v'\rc^{(t)}\big)\\
&=&\lbar{\varphi}\big(\lc u'  \diamond \lc v'\rc\rc^{(s+t-1)}\big)\hspace{1.3cm}(\text{by  ~\meqref{eq:bre1}})\\
&=&\lbar{\varphi}\big(\A_X^{s+t-1}( u'  \diamond \A_X( v'))\big)\hspace{1.3cm}(\text{by  ~\meqref{eq:daveo}})\\
&=&\A^{s+t-1}\lbar{\varphi}\big( u'  \diamond \A_X( v')\big)\hspace{1.5cm}(\text{by  ~\meqref{eq:rbdef1}})\\
&=&\A^{s+t-1}\Big(\lbar{\varphi}( u')  \lbar{\varphi}\big( \A_X( v')\big)\Big)\hspace{1cm}(\text{by the induction hypothesis})\\
&=&\A^{s+t-1}\Big(\lbar{\varphi}( u')  \A\big(\lbar{\varphi} ( v')\big)\Big)\\
&=&\A^{s}\Big(\lbar{\varphi}( u')\Big)  \A^t\Big(\lbar{\varphi} ( v')\Big)\\
&=&\lbar{\varphi}\big(\A_X^{s}( u')\big) \lbar{\varphi}\big(\A_X^{t}( v')\big)\\
&=&\lbar{\varphi}\big(\lc u'\rc^{(s)}\big) \lbar{\varphi}\big(\lc v'\rc^{(t)}\big)\\
&=&\lbar{\varphi}\big(u\big) \lbar{\varphi}\big(v\big).
\end{eqnarray*}

\noindent{\bf Case 2.} Suppose that $\bre(u)=\bre(v)=1$, $u= \big(\lc u' \rc^{(p)}\big)^{-1}$ and $v=\big(\lc v' \rc^{(q)}\big)^{-1}$ for some $u',v'\in \rbgw$.
Using the induction hypothesis,
\begin{eqnarray*}
  \lbar{\varphi}(u\diamond v)
 &=& \lbar{\varphi}\Big(\big(\lc u' \rc^{(p)}\big)^{-1} \diamond \big(\lc v' \rc^{(q)}\big)^{-1}\Big) \\
&=& \lbar{\varphi}\Big(\big(\lc {v'}\rc^{(q)}\diamond \lc {u'}\rc^{(p)}\big)^{-1}\Big) \\
&=& \Big(\lbar{\varphi}\big(\lc {v'} \rc^{(q)}\diamond \lc{u'}\rc^{(p)}\big)\Big)^{-1} \hspace{2cm} (\text{by  ~(\ref{eq:rbdef2})})\\
&=& \Big(\lbar{\varphi}(\lc {v'} \rc^{(q)}) \,\lbar{\varphi}(\lc{u'}\rc^{(p)})\Big)^{-1} \hspace{2cm} (\text{by Case 1})\\
&=&  \Big(\lbar{\varphi}\big(\lc {u'} \rc^{(p)}\big)\Big)^{-1} \,\Big(\lbar{\varphi}\big(\lc{v'}\rc^{(q)}\big)\Big)^{-1}\\
&=& \lbar{\varphi}\Big((\lc {u'} \rc^{(p)})^{-1}\Big) \,\lbar{\varphi}\Big((\lc{v'}\rc^{(q)})^{-1}\Big) \hspace{2cm} (\text{by  ~(\ref{eq:rbdef2})}) \\
&=& \lbar{\varphi}(u) \lbar{\varphi}(v).
\end{eqnarray*}

\noindent{\bf Case 3.} Suppose $\bre(u) = \bre(v)=1$, but $u, v$ are not in the previous two cases. Then by definition, $u\diamond v= uv$ is the concatenation. It follows from  ~(\mref{eq:rbdef3}) that
\begin{align*}
\lbar{\varphi}(u\diamond v)=&\ \lbar{\varphi}(uv)= \lbar{\varphi}(u)\lbar{\varphi}(v).
\end{align*}

\noindent{\bf Case 4.} Suppose $\bre(u)+\bre(v)\geq3$. We have at least one of $\bre(u)$ or $\bre(v)$ is greater than $1$. Let $u=u_{1}\cdots u_{m}$ and $v=v_{1}\cdots v_{n}$ be in the standard factorizations in  ~\meqref{eq:frbg}.
Then
\begin{eqnarray*}
&& \lbar{\varphi}(u\diamond v)\\
&=& \lbar{\varphi}\big(u_{1}\cdots u_{m-1}(u_{m}\diamond v_{1})v_{2}\cdots v_{n}\big) \\
&=&\lbar{\varphi}(u_{1})\cdots \lbar{\varphi}(u_{m-1})\, \lbar{\varphi}(u_{m}\diamond v_{1}) \,\lbar{\varphi}(v_{2}) \cdots \lbar{\varphi}(v_{n}) \hspace{2cm} (\text{by~(\ref{eq:rbdef3})})\\
&=& \lbar{\varphi}(u_{1})\cdots \lbar{\varphi}(u_{m})\, \lbar{\varphi}(v_{1})\cdots \lbar{\varphi}(v_{n}) \hspace{4cm} (\text{by Cases 1, 2, 3})\\
&=& \lbar{\varphi}(u_{1}\cdots  u_{m}) \, \lbar{\varphi}(v_{1}\cdots v_{n}) \hspace{5cm}(\text{by~(\ref{eq:rbdef3})})\\
&=& \lbar{\varphi}(u)\lbar{\varphi}(v).
\end{eqnarray*}

Therefore the map $\lbar{\varphi}$ is an \aveg morphism.

({\bf Uniqueness}) Notice that   ~(\ref{eq:bfdep0})--(\ref{eq:rbdef3}) give the only possible way to define $\lbar{\varphi}(w)$
in order for $\lbar{\varphi}$ to be an \aveg homomorphism that extends $\varphi$. This proves the uniqueness.
\end{proof}

\noindent
{\bf Acknowledgments.}
The second author is supported by
the Natural Science Foundation of Gansu Province (25JRRA644), Innovative Fundamental Research Group Project of Gansu Province (23JRRA684), and Longyuan Young Talents of Gansu Province.
The first author is supported by the Scientific Research Foundation of High-Level Talents of Yulin University and Young Talent Fund of Association for Science and Technology in Shaanxi, China (20250530).

\noindent
{\bf Competing Interests.} On behalf of all authors, the corresponding author states that there is no conflict of interest.

\noindent
{\bf Data Availability.} The manuscript has no associated data.

\end{document}